\documentclass[12pt,a4paper]{amsart}

\usepackage{amssymb,verbatim}
\usepackage{amsfonts}
\usepackage[utf8]{inputenc}
\usepackage[mathscr]{euscript}
\usepackage{enumerate}
\usepackage{hyperref}

\usepackage{tkz-graph}
\usetikzlibrary{arrows}
\usetikzlibrary{positioning}
\usepackage[normalem]{ulem}

\usepackage[
margin=1in,
includefoot,
footskip=30pt,
]{geometry}

\usepackage{mathtools}
\usepackage{xcolor}

\newtheorem{theorem}[equation]{Theorem}
\newtheorem{lemma}[equation]{Lemma}
\newtheorem{proposition}[equation]{Proposition}
\newtheorem{corollary}[equation]{Corollary}

\theoremstyle{definition}
\newtheorem{definition}[equation]{Definition}
\newtheorem{example}[equation]{Example}

\theoremstyle{remark}
\newtheorem{remark}[equation]{Remark}

\numberwithin{equation}{section}

\newcommand{\vecspan}{\operatorname{span}}
\newcommand{\la}[1]{\text{$\mathcal{#1}$}}
\newcommand{\lb}[1]{\text{$\mathscr{#1}$}}

\def\Z{\mathbb Z}

\newcommand{\powerset}[1]{\text{$\lb{P}(#1)$}}								
\newcommand{\eword}{\text{$\omega$}}											

\newcommand{\xia}{\text{$\xi^\alpha$}}											

\newcommand{\ph}{\text{$\hat{\phi}$}}											

\newcommand{\usetr}[2]{\text{$\uparrow_{\scriptscriptstyle{#2}}\hspace{-0.1cm}{#1}$}} 

\newcommand{\dgraphg}[1]{\text{$\lb{#1}$}}									
\newcommand{\dgraphupleg}[3]{\text{$\dgraphg{#1}=(\dgraphg{#1}^0,\dgraphg{#1}^1,#2,#3)$}}	

\newcommand{\dgraph}{\text{$\dgraphg{E}$}}									
\newcommand{\dgraphuple}{\text{$\dgraphupleg{E}{r}{s}$}}					
\newcommand{\alfg}[1]{\text{$\lb{#1}$}}										
\newcommand{\acfg}[1]{\text{$\lb{#1}$}}										
\newcommand{\lbfg}[1]{\text{$\lb{#1}$}}										
\newcommand{\acfrg}[1]{\text{$\acfg{B}_{#1}$}}								

\newcommand{\alf}{\text{$\alfg{A}$}}											
\newcommand{\acf}{\text{$\acfg{B}$}}											
\newcommand{\lbf}{\text{$\lbfg{L}$}}											
\newcommand{\acfra}{\text{$\acfg{B}_{\alpha}$}}								
\newcommand{\lgraphg}[2]{\text{$(\dgraphg{#1},\lbfg{#2})$}}				
\newcommand{\lspaceg}[3]{\text{$(\dgraphg{#1},\lbfg{#2},\acfg{#3})$}}		
\newcommand{\lgraphgi}[3]{\text{$(\dgraphg{#1}_{#3},\lbfg{#2}_{#3})$}}

\newcommand{\lgraph}{\text{$\lgraphg{E}{L}$}}								
\newcommand{\lspace}{\text{$\lspaceg{E}{L}{B}$}}							
\newcommand{\lspacei}[1]{\text{$(\dgraphg{E}_{#1},\lbfg{L}_{#1},\acfg{B}_{#1})$}}

\newcommand{\awsetg}[2]{\text{$\lbfg{#1}^{#2}$}}	
\newcommand{\awplus}{\text{$\awsetg{L}{\scriptscriptstyle{\geq 1}}$}}								
\newcommand{\awn}[1]{\text{$\awsetg{L}{#1}$}}								
\newcommand{\awstar}{\text{$\awsetg{L}{\ast}$}}							
\newcommand{\awinf}{\text{$\awsetg{L}{\infty}$}}							
\newcommand{\awleinf}{\text{$\awsetg{L}{\scriptscriptstyle{\leq\infty}}$}}					
\newcommand{\F}{\mathbb{F}}
\newcommand {\N}{\mathbb{N}}

\newcommand{\ftight}{{\normalfont \textsf{T}}}										

\newcommand{\ftg}[1]{\text{$\la{#1}$}}											
\newcommand{\ft}{\text{$\ftg{F}$}}												

\newcommand{\scj}{\subseteq}

\newcommand{\nn}{\mathbb{N}}
\newcommand{\zn}{\mathbb{Z}}

\newcommand{\Lc}{\operatorname{Lc}}
\newcommand{\Ann}{\operatorname{Ann}}

\makeatletter
\@namedef{subjclassname@2020}{%
  \textup{2020} Mathematics Subject Classification}
\makeatother

\allowdisplaybreaks

\begin{document}

\title{Leavitt path algebras of labelled graphs}

\author[G. Boava]{Giuliano Boava}
\author[G.G. de Castro]{Gilles G. de Castro}
\author[D. Gonçalves]{Daniel Gonçalves}
\address[Giuliano Boava, Gilles G. de Castro and Daniel Gonçalves]{Departamento de Matem\'atica, Universidade Federal de Santa Catarina, 88040-970 Florian\'opolis SC, Brazil. }
\email{g.boava@ufsc.br \\ gilles.castro@ufsc.br \\ daemig@gmail.com}
\author[D.W. van Wyk]{Daniel W. van Wyk}
\address[Daniel W. van Wyk]{Department of Mathematics, Dartmouth College, Hanover, NH 03755-3551 USA.}
\email{daniel.w.van.wyk@dartmouth.edu }


\thanks{The second and third authors were partially supported by Capes-Print Brazil. The third author was partially supported by CNPq - Conselho Nacional de Desenvolvimento Científico e Tecnológico - Brazil.
 }

\begin{abstract}
A Leavitt labelled path algebra over a commutative unital ring is associated with a labelled space, generalizing Leavitt path algebras associated with graphs and ultragraphs as well as torsion-free commutative algebras generated by idempotents. We show that Leavitt labelled path algebras can be realized as partial skew group rings, Steinberg algebras, and Cuntz-Pimsner algebras. Via these realizations we obtain generalized uniqueness theorems, a description of diagonal preserving isomorphisms and we characterize simplicity of Leavitt labelled path algebras. In addition, we prove that a large class of partial skew group rings can be realized as Leavitt labelled path algebras.

\end{abstract}

\maketitle

\section{Introduction}\label{section:introduction}

Leavitt path algebras were independently defined in \cite{AAP} and \cite{AGG} and, due to their high versatility, have grown into an important area of algebra. Among applications of Leavitt path algebras, we highlight that they have deep connections with symbolic dynamics and the theory of graph C*-algebras. For example, the notion of flow equivalence of shifts of finite type in symbolic dynamics is related to  Morita theory and the Grothendieck group in the theory of Leavitt path algebras, see  \cite{ALP, Hazrat13}; and ring isomorphism (or Morita equivalence) between two Leavitt path algebras over the field of the complex numbers induces, for some graphs, isomorphism (or Morita equivalence) of the respective graph C*-algebras (see \cite{AT11, RTh}). For a detailed introduction to the theory of Leavitt path algebras, we refer the reader to \cite{AbrAraMol}.

In the C*-setting, various generalizations of graph C*-algebras have been constructed (such as ultragraph, higher-rank graph, and labelled graph C*-algebras). The connection of these C*-algebras with the associated intrinsic shift spaces has been explored deeply, for example, see \cite{MR3938320}, where conjugacy that preserves length of ultragraph shift spaces is described in terms of isomorphisms of the associated ultragraph C*-algebras, in \cite{TokeRout}, where continuous orbit equivalence, eventual one-sided conjugacy, and two-sided conjugacy of higher-rank graphs is characterized in terms of the C*-algebras (and the Kumjian-Pask algebras) of the higher-rank graphs, and in \cite{MR3614028}, where the K-theory of the C*-algebra associated with a labelled spaces is computed. In the algebraic setting, versions of ultragraph and higher-rank graphs algebras have been defined, see \cite{imanfar2017leavitt, ACHR}, but an algebraic analogue of labelled graph C*-algebras is still missing. Our main goal in this paper is to introduce Leavitt path algebras associated with labelled graphs, which we will call Leavitt labelled path algebras, and to start their study.

Labelled graphs are the key ingredient to construct sofic shifts, which form the smallest collection of shift spaces that contains all shifts of finite type and also contains all factors of
each space in the collection (see \cite{LindMarcus}). The connection between labelled graph C*-algebras and symbolic dynamics is explored in \cite{MR3614028, BP1, GillesDanie}, for example. In our work, we start to describe relations between the intrinsic dynamics and combinatorics of a labelled graph (in fact a labelled space, as we shall soon see) and an associated Leavitt labelled path algebra. For this, we will use a partial skew group ring, a groupoid (Steinberg) algebra, and a Cuntz-Pimsner ring description of the Leavitt labelled path algebras we define, and will then interpret, in our context, general results (about groupoid algebras, partial skew group rings, and Cuntz-Pimsner rings) that allow us to trace finer structures of the algebras back to properties of the associated dynamical systems and vice-versa.
 
A key ingredient to the above descriptions is the Graded Uniqueness Theorem. In contrast to Leavitt path algebras of graphs and ultragraphs (see \cite{AbrAraMol,imanfar2017leavitt}), for which the Graded Uniqueness Theorem is proved before any realization, in this paper, we first prove that a Leavitt labelled path algebra can be realized as a Cuntz-Pimsner ring, and then use the Graded Uniqueness Theorem available to Cuntz-Pimsner rings.
 
Going back one last time to the C*-algebraic setting, we recall that labelled graph C*-algebras provide a common framework for ultragraph C*-algebras and the C*-algebras associated with shift spaces by Matsumoto and Carlsen (see \cite{BP1}). Our algebraic definition of labelled graph algebras also includes ultragraph Leavitt path algebras but, as far as we know, algebraic analogues of Matsumoto and Carlsen algebras have not been defined yet, and we intend to explore them in the future. Moreover, the Leavitt labelled path algebras we define include partial skew group rings associated with partial actions by the free group on commutative, torsion-free algebras generated by idempotents. This last result is a partial converse to the realization of Leavitt labelled path algebras as partial skew group rings mentioned in the previous paragraph and, as far as we know, there is no analytical analogue to it.

Before we proceed, we notice that when studying algebras associated with a labelled graph, one has to be aware of a key point: usually there is more than one algebra that can be associated with a given labelled graph. This leads to the notion of a labelled space and its associated algebra. A labelled space is a triple $\lspace$, where $\lgraph$ is a labelled graph and $\acf$ is a family of subsets of $\dgraph^0$ that satisfy certain conditions (see Section~\ref{labelled?}). For each labelled space (that is also normal) and each unital commutative ring $R$, we associate an algebra over $R$, which is given by generators and relations (see Definition~\ref{def:LPA}), and which is the focus of the paper.

We now give a more detailed overview of the paper.

In Section~\ref{section:preliminaries}, we provide the reader with the necessary preliminaries on labelled graphs, labelled spaces, and the inverse semigroup associated with a labelled space, following \cite{BP1} and \cite{BoavaDeCastroMortari1}. 

In Section~\ref{section:leavitt.labelled.path.algebras}, we define the Leavitt labelled path algebra associated with a normal labelled space and describe some of its properties, including a $\Z$-grading. 

Using the tight spectrum of the inverse semigroup associated with a labelled space, and building on ideas in \cite{GillesDanie}, in Section~\ref{section:partial.skew.group}, we describe a partial action associated with a normal labelled space and construct its associated partial skew group ring. Moreover, we prove that there is a surjective homomorphism from the Leavitt labelled path algebra onto the partial skew group ring (Proposition~\ref{prop:OntoHomLeavittPartialSkew}).

To prove injectivity of the aforementioned homomorphism, we dedicate Section~\ref{section:graded.uniqueness.theorem} to obtain a Graded Uniqueness Theorem (Corollary~\ref{GUT}). For this, we realize the Leavitt labelled path algebra associated with a normal labelled space as a Cuntz-Pimsner ring (Theorem~\ref{LPACPR}), and then apply results in \cite{CO55} to get the desired Graded Uniqueness Theorem (Theorem~\ref{GUT}), which characterizes when graded homomorphisms from the Leavitt labelled path algebra are injective. As corollaries we get the description of a Leavitt labelled path algebra as a partial skew group ring (Theorem~\ref{thm:SkewRingLeavittIsom}) and the characterization of continuous orbit equivalence between the partial actions associated with two labelled spaces in terms of an isomorphism that preserves the diagonal subalgebras of the respective partial skew group rings (Corollary~\ref{orb1qnob}). 

Proceeding, in Section~\ref{section:steinberg.algebra}, we use the partial skew group ring description of Leavitt labelled path algebras to obtain, via the transformation groupoid, the realization of Leavitt labelled path algebras as Steinberg (groupoid) algebras (Theorem~\ref{thm:LeavittSteinbergIsom}). As consequences, we describe when a Leavitt labelled path algebra is unital (Corollary~\ref{unital}) and describe diagonal preserving isomorphisms of these algebras, see Corollary~\ref{orb2qnob} (which is further improved in Corollary~\ref{orb1qnob}).

We devote Section~\ref{section:examples} to examples of Leavitt labelled path algebras. In particular, we show that Leavitt path algebras of graphs and ultragraphs can be seen as Leavitt labelled path algebras. We also give sufficient conditions under which the converse holds, that is, under which a Leavitt labelled path algebra can be seen as a graph Leavitt path algebra (Proposition~\ref{LLPALPA}). Furthermore, we prove that every torsion-free commutative algebra generated by idempotents is isomorphic to a Leavitt labelled path algebra. These provide examples of Leavitt labelled path algebras that are not isomorphic to any Leavitt path algebra of a graph (or an ultragraph). It is also in this section that we show that a class of partial skew group rings associated with partial actions by the free group on commutative, torsion-free algebras generated by idempotents can be realized as Leavitt labelled graph algebras, see Corollary~\ref{jacarevoador}. The reader might want to read this section right after Section~\ref{section:leavitt.labelled.path.algebras}, to get a feel of the algebras associated with labelled spaces, but some tools necessary to prove the results we just mentioned are only developed in the sections following Section~\ref{section:leavitt.labelled.path.algebras}.

Next, in Section~\ref{section:general.uniqueness.theorems}, we describe a general uniqueness theorem, which does not rely on the grading of the algebras nor on the combinatorial structure of the labelled graph. Our main result in this section states that a homomorphism from a Leavitt labelled path algebra is injective if, and only if, it is injective in the abelian core subalgebra (Theorem~\ref{GUTA}). This is done by describing the isotropy bundle of the transformation groupoid and showing that the abelian core subalgebra is isomorphic to the Steinberg algebra of the interior of the isotropy bundle (Proposition~\ref{p:abcore}).

Finally, we finish the paper in Section~\ref{section:simplicity} using the Steinberg algebra realization of Leavitt labelled path algebras to characterize simplicity of these algebras in terms of the hereditary and saturated subsets of the underlying graph and the fulfillment of Condition~$(L_\acf)$ (which is a type of Condition~(L) for labelled graphs), see Theorem~\ref{theorem:simple.labelled.space.algebra}.

\section{Preliminaries}\label{section:preliminaries}

Throughout this paper we let $\N$ and $\N^*$ denote the set of non-negative integers, and positive integers, respectively.

\subsection{Labelled graphs and labelled spaces}\label{labelled?}

A \emph{(directed) graph} $\dgraphuple$ consists of non-empty sets $\dgraph^0$ (of \emph{vertices}), $\dgraph^1$ (of \emph{edges}), and \emph{range} and \emph{source} functions $r,s:\dgraph^1\to \dgraph^0$. A vertex $v$ is called a \emph{sink} if $s^{-1}(v)=\emptyset$, and it is called an \emph{infinite emitter} if $s^{-1}(v)$ is an infinite set. The set of all sinks is denoted by $\dgraph^0_{sink}$. 

A \emph{path of length $n$} on a graph $\dgraph$ is a sequence $\lambda=\lambda_1\lambda_2\ldots\lambda_n$ of edges such that $r(\lambda_i)=s(\lambda_{i+1})$ for all $i=1,\ldots,n-1$. We write $|\lambda|=n$ for the length of $\lambda$ and regard vertices as paths of length $0$. $\dgraph^n$ stands for the set of all paths of length $n$ and $\dgraph^{\ast}=\cup_{n\geq 0}\dgraph^n$. Similarly, we define a \emph{path of infinite length} (or an \emph{infinite path}) as an infinite sequence $\lambda=\lambda_1\lambda_2\ldots$ of edges such that $r(\lambda_i)=s(\lambda_{i+1})$ for all $i\geq 1$; for such a path, we write $|\lambda|=\infty$ and we let $\dgraph^{\infty}$ denote the set of all infinite paths.

A \emph{labelled graph} consists of a graph $\dgraph$ together with a surjective \emph{labelling map} $\lbf:\dgraph^1\to\alf$, where $\alf$ is a fixed non-empty set, called an \emph{alphabet}, and whose elements are called \emph{letters}. $\alf^{\ast}$ stands for the set of all \emph{finite words} over $\alf$, together with the \emph{empty word} \eword, and $\alf^{\infty}$ is the set of all \emph{infinite words} over $\alf$. We consider $\alf^{\ast}$ as a monoid with operation given by concatenation. In particular, given $\alpha\in \alf^{\ast}\setminus\{\eword\}$ and $n\in\N^*$, $\alpha^n$ represents $\alpha$ concatenated $n$ times and $\alpha^{\infty}\in \alf^{\infty}$ is $\alpha$ concatenated infinitely many times.

The labelling map $\lbf$ extends in the obvious way to $\lbf:\dgraph^n\to\alf^{\ast}$ and $\lbf:\dgraph^{\infty}\to\alf^{\infty}$. Let $\awn{n}=\lbf(\dgraph^n)$ be the set of \emph{labelled paths $\alpha$ of length $|\alpha|=n$}, and $\awinf=\lbf(\dgraph^{\infty})$ is the set of \emph{infinite labelled paths}. We consider $\eword$ as a labelled path with $|\eword|=0$, and set $\awplus=\cup_{n\geq 1}\awn{n}$, $\awstar=\{\eword\}\cup\awplus$, and $\awleinf=\awstar\cup\awinf$.  

For $\alpha\in\awstar$ and $A\in\powerset{\dgraph^0}$ (the power set of $\dgraph^0$), the \emph{relative range of $\alpha$ with respect to} $A$ is the set
\[r(A,\alpha)=
\begin{cases}
\{r(\lambda)\ |\ \lambda\in\dgraph^{\ast},\ \lbf(\lambda)=\alpha,\ s(\lambda)\in A\}, &
\text{if }\alpha\in\awplus \\ 
A, & \text{if }\alpha=\eword.
\end{cases}
\]
The \emph{range of $\alpha$}, denoted by $r(\alpha)$, is the set \[r(\alpha)=r(\dgraph^0,\alpha),\]
so that $r(\eword)=\dgraph^0$ and, if $\alpha\in\awplus$, then  $r(\alpha)=\{r(\lambda)\in\dgraph^0 \mid \lbf(\lambda)=\alpha\}$. We also define \[\lbf(A\dgraph^1)=\{\lbf(e) \mid e\in\dgraph^1\ \mbox{and}\ s(e)\in A\}=\{a\in\alf \mid r(A,a)\neq\emptyset\}.\]

A labelled path $\alpha$ is a \emph{beginning} of a labelled path $\beta$ if $\beta=\alpha\beta'$ for some labelled path $\beta'$. Labelled paths  $\alpha$ and $\beta$ are \emph{comparable} if either one is a beginning of the other. If $1\leq i\leq j\leq |\alpha|$, let $\alpha_{i,j}=\alpha_i\alpha_{i+1}\ldots\alpha_{j}$ if $j<\infty$ and $\alpha_{i,j}=\alpha_i\alpha_{i+1}\ldots$ if $j=\infty$. If $j<i$ set $\alpha_{i,j}=\eword$. Define $\overline{\awinf}=\overline{\lbf(\dgraph^{\infty})}=\{\alpha\in\alf^{\infty}\mid \alpha_{1,n}\in\awstar,\forall n\in\N\}$, that is, it is the set of all infinite words such that all beginnings are finite labelled paths. Also we write $\overline{\awleinf}=\awstar\cup\overline{\awinf}$.

A \emph{labelled space} is a triple $\lspace$ where $\lgraph$ is a labelled graph and $\acf$ is a family of subsets of $\dgraph^0$ which is closed under finite intersections and finite unions, contains  $r(\alpha)$ for every $\alpha\in\awplus$, and is \emph{closed under relative ranges}, that is, $r(A,\alpha)\in\acf$ for all $A\in\acf$ and all $\alpha\in\awstar$. The family of sets $\acf$ is called an \textit{accommodating family} for $\lgraph$. A labelled space $\lspace$ is \emph{weakly left-resolving} if for all $A,B\in\acf$ and all $\alpha\in\awplus$ we have $r(A\cap B,\alpha)=r(A,\alpha)\cap r(B,\alpha)$.  A weakly left-resolving labelled space such that $\acf$ is closed under relative complements will be called \emph{normal}\footnote{
	Note this definition differs from \cite{MR3614028}. Since all labelled spaces considered in this paper are weakly left-resolving, we include `weakly left-resolving' in the
	definition of a  normal labelled space. }. A non-empty set $A\in\acf$ is called \emph{regular} if for all $\emptyset\neq B\scj A$, we have that $0<|\lbf(B\dgraph^1)|<\infty$. The subset of all regular element of $\acf$ together with the empty set is denoted by $\acf_{reg}$.

For $\alpha\in\awstar$, define \[\acfra=\acf\cap\powerset{r(\alpha)}=\{A\in\acf\mid A\scj r(\alpha)\}.\] If a labelled space is normal, then  $\acfra$ is a Boolean algebra for each $\alpha\in\awstar$.

\subsection{The inverse semigroup of a labelled space}

For a weakly left-resolving labelled space $\lspace$, consider the set \[S=\{(\alpha,A,\beta)\mid \alpha,\beta\in\awstar\ \mbox{and}\ A\in\acfrg{\alpha}\cap\acfrg{\beta}\ \mbox{with}\ A\neq\emptyset\}\cup\{0\}.\]
Define a  binary operation on $S$  as follows: $s\cdot 0= 0\cdot s=0$ for all $s\in S$ and, if $s=(\alpha,A,\beta)$ and $t=(\gamma,B,\delta)$ are in $S$, then \[s\cdot t=\left\{\begin{array}{ll}
(\alpha\gamma ',r(A,\gamma ')\cap B,\delta), & \mbox{if}\ \  \gamma=\beta\gamma '\ \mbox{and}\ r(A,\gamma ')\cap B\neq\emptyset,\\
(\alpha,A\cap r(B,\beta '),\delta\beta '), & \mbox{if}\ \  \beta=\gamma\beta '\ \mbox{and}\ A\cap r(B,\beta ')\neq\emptyset,\\
0, & \mbox{otherwise}.
\end{array}\right. \]
For  $s=(\alpha,A,\beta)\in S$, we define $s^*=(\beta,A,\alpha)$. 

The set $S$ endowed with the operations above is an inverse semigroup with zero element $0$ (\cite{BoavaDeCastroMortari1}, Proposition 3.4), whose semilattice of idempotents is \[E(S)=\{(\alpha, A, \alpha) \mid \alpha\in\awstar \ \mbox{and} \ A\in\acfra\}\cup\{0\}.\]

The natural order in the semilattice $E(S)$ is given by $p\leq q$ if only if $pq=p$. In our case, the order can described by noticing that if $p=(\alpha, A, \alpha)$ and $q=(\beta, B, \beta)$, then $p\leq q$ if and only if $\alpha=\beta\alpha'$ and $A\scj r(B,\alpha')$ (see \cite[ Proposition 4.1]{BoavaDeCastroMortari1}).

\subsection{Filters in \texorpdfstring{$E(S)$}{E(S)}}\label{subsection:filters.E(S)}
In what follows we need to describe of filters in $E(S)$, as well as the set of tight filters, denote by $\ftight$ and called the \emph{tight spectrum}. For the definition of tight filters, see \cite[Section 12]{MR2419901}.

Let $\alpha\in\overline{\awleinf}$  and $\{\ftg{F}_n\}_{0\leq n\leq|\alpha|}$ (understanding that ${0\leq n\leq|\alpha|}$ means $0\leq n<\infty$ when $\alpha\in\overline{\awinf}$) be a family such that $\ftg{F}_n$ is a filter in $\acfrg{\alpha_{1,n}}$ for every $n>0$, and $\ftg{F}_0$ is either a filter in $\acf$ or $\ftg{F}_0=\emptyset$. The family $\{\ftg{F}_n\}_{0\leq n\leq|\alpha|}$ is  a  \emph{complete family for} $\alpha$ if
\[\ftg{F}_n = \{A\in \acfrg{\alpha_{1,n}} \mid r(A,\alpha_{n+1})\in\ftg{F}_{n+1}\}\]
for all $n\geq0$.

\begin{theorem}\cite[Theorem 4.13]{BoavaDeCastroMortari1}\label{thm.filters.in.E(S)}
	Let $\lspace$ be a weakly left-resolving labelled space and $S$   its associated inverse semigroup. Then there is a bijective correspondence between filters in $E(S)$ and pairs $(\alpha, \{\ftg{F}_n\}_{0\leq n\leq|\alpha|})$, where $\alpha\in\overline{\awleinf}$ and $\{\ftg{F}_n\}_{0\leq n\leq|\alpha|}$ is a complete family for $\alpha$.
\end{theorem}

Filters are of \emph{finite type} if they are associated with pairs $(\alpha, \{\ftg{F}_n\}_{0\leq n\leq|\alpha|})$ for which $|\alpha|<\infty$, and of \emph{infinite type} otherwise.

A filter $\xi$ in $E(S)$ with associated labelled path $\alpha\in\overline{\awleinf}$ is sometimes denoted by $\xia$ to stress the word $\alpha$; in addition, the filters in the complete family associated with $\xia$ will be denoted by $\xi^{\alpha}_n$ (or simply $\xi_n$). Specifically,
\begin{align}\label{eq.defines.xi_n}
\xi^{\alpha}_n=\{A\in\acf \mid (\alpha_{1,n},A,\alpha_{1,n}) \in \xia\}.
\end{align}

\begin{remark}\label{remark.when.in.xialpha}
	It follows from \cite[Propositions 4.4 and 4.8]{BoavaDeCastroMortari1} that for a filter $\xi^\alpha$ in $E(S)$ and an element $(\beta,A,\beta)\in E(S)$ we have that $(\beta,A,\beta)\in\xi^{\alpha}$ if and only if $\beta$ is a beginning of $\alpha$ and $A\in\xi_{|\beta|}$.
\end{remark}

\begin{theorem}[\cite{BoavaDeCastroMortari1}, Theorems 5.10 and 6.7] \label{thm:TightFiltersType}
	Let $\lspace$ be a normal labelled space and $S$  its associated inverse semigroup. Then the tight filters in $E(S)$ are:
	\begin{enumerate}[(i)]
		\item The filters of infinite type for which the non-empty elements of their associated complete families are ultrafilters. 
		\item The filters of finite type $\xia$ such that $\xi_{|\alpha|}$ is an ultrafilter in $\acfra$ and for each  $A\in\xi_{|\alpha|}$ at least one of the following conditions hold:
		\begin{enumerate}[(a)]
			\item $\lbf(A\dgraph^1)$ is infinite.
			\item There exists $B\in\acfra$ such that $\emptyset\neq B\scj A\cap \dgraph^0_{sink}$.
		\end{enumerate}
	\end{enumerate}
\end{theorem}

\section{Leavitt labelled path algebras}\label{section:leavitt.labelled.path.algebras}

In this section, we define an algebra associated with a normal labelled space and we prove some of its initial properties. In later sections, we give different descriptions of these algebras and we prove that they generalize Leavitt path algebras.

\begin{definition}\label{def:LPA}
Let $\lspace$ be a normal labelled space and $R$ a unital commutative ring. The \emph{Leavitt labelled path algebra associated with $\lspace$ with coefficients in $R$}, denoted by $L_R\lspace$, is the universal $R$-algebra with generators $\{p_A \mid A\in \acf\}$ and $\{s_a,s_a^* \mid a\in\alf\}$ subject to the relations
\begin{enumerate}[(i)]
	\item $p_{A\cap B}=p_Ap_B$, $p_{A\cup B}=p_A+p_B-p_{A\cap B}$ and $p_{\emptyset}=0$, for every $A,B\in\acf$;
	\item $p_As_a=s_ap_{r(A,a)}$ and $s_a^{*}p_A=p_{r(A,a)}s_a^{*}$, for every $A\in\acf$ and $a\in\alf$;
	\item $s_a^*s_a=p_{r(a)}$ and $s_b^*s_a=0$ if $b\neq a$, for every $a,b\in\alf$;
	\item $s_as_a^*s_a=s_a$ and $s_a^{*}s_as_a^{*}=s_a^{*}$ for every $a\in\alf$;
	\item For every $A\in\acf_{reg}$,
	\[p_A=\sum_{a\in\lbf(A\dgraph^1)}s_ap_{r(A,a)}s_a^*.\]
\end{enumerate}
\end{definition}

When the labelled space and the ring are clear in a given context, we call $L_R\lspace$ a Leavitt labelled path algebra.

For each word $\alpha=a_1a_2\cdots a_n$, define $s_\alpha=s_{a_1}\cdots s_{a_n}$ and $s_\alpha^{*}=s_{a_n}^{*}\cdots s_{a_1}^{*}$. Although $L_R\lspace$ is not necessarily unital, we also set $s_{\eword}=s_{\eword}^*=1$, where $\eword$ is the empty word. This is done only in order to simplify certain statements. For example, $s_{\eword}p_{A}s_{\eword}^*$ means $p_{A}$. We never use $s_{\eword}$ by itself. 

Fix a normal labelled space $\lspace$ with and a unital ring $R$.
\begin{proposition}\label{prop:properties.of.c*-labelled.space}
	The following properties are valid in $L_R\lspace$.
	\begin{enumerate}[(i)]
		\item If $\alpha\notin\awstar$, then $s_{\alpha}=0$.
		\item $p_As_\alpha=s_\alpha p_{r(A,\alpha)}$ and $s_\alpha^{*}p_A= p_{r(A,\alpha)}s_\alpha^{*}$, for every $A\in\acf$ and $\alpha\in\awstar$.
		\item $s_\alpha^*s_\alpha=p_{r(\alpha)}$ and $s_\beta^*s_\alpha=0$ if $\beta$ and $\alpha$ are not comparable, for every $\alpha,\beta\in\awplus$.
		\item For every $\alpha\in\awplus$, $s_{\alpha}s_{\alpha}^{*}s_{\alpha}=s_{\alpha}$ and $s_{\alpha}^{*}s_{\alpha}s_{\alpha}^{*}=s_{\alpha}^{*}$
		\item Let $\alpha,\beta\in\awstar$ and $A\in\acf$. If $s_\alpha p_A s_\beta^*\neq 0$, then $A\cap r(\alpha)\cap r(\beta)\neq\emptyset$ and $s_\alpha p_A s_\beta^* = s_\alpha p_{A\cap r(\alpha)\cap r(\beta)} s_\beta^*$. 
		\item Let $\alpha,\beta,\gamma,\delta\in\awstar$, $A\in\acfrg{\alpha}\cap\acfrg{\beta}$ and $B\in\acfrg{\gamma}\cap\acfrg{\delta}$. Then
		\[(s_\alpha p_A s_\beta^*)(s_\gamma p_B s_\delta^*) = \left\{ 
		\begin{array}{ll}
		s_{\alpha\gamma'}p_{r(A,\gamma')\cap B}s_{\delta}^*, & \mbox{if} \ \gamma=\beta\gamma', \\
		s_{\alpha}p_{A\cap r(B,\beta')}s_{\delta\beta'}^*, & \mbox{if} \ \beta=\gamma\beta', \\
		0, & \mbox{otherwise}. 
		\end{array}\right.\]
		\item Every non-zero finite product of terms of types $s_a$, $p_B$ and $s_b^*$ can be written as $s_\alpha p_A s_\beta^*$, where $A\in \acfrg{\alpha}\cap\acfrg{\beta}$.
		\item $L_R\lspace = \vecspan_R\{s_\alpha p_A s_\beta^* \mid \alpha,\beta\in\awstar \ \mbox{and} \ A\in\acfrg{\alpha}\cap\acfrg{\beta}\}$.
		\item Elements of the form $s_\alpha p_A s_\alpha^*$, where $\alpha\in\awstar$, are commuting idempotents. Furthermore,
		\[(s_\alpha p_A s_\alpha^*)(s_\beta p_B s_\beta^*) = \left\{ 
		\begin{array}{ll}
		s_{\beta}p_{r(A,\beta')\cap B}s_{\beta}^*, & \mbox{if} \ \beta=\alpha\beta', \\
		s_{\alpha}p_{A\cap r(B,\alpha')}s_{\alpha}^*, & \mbox{if} \ \alpha=\beta\alpha', \\
		0, & \mbox{otherwise}. 
		\end{array}\right.\]
	\end{enumerate}
\end{proposition}

\begin{proof}
The proof follows the same line of reasoning as its C*-algebraic counterpart; see \cite[Proposition 3.4]{BoavaDeCastroMortari2}. 
\end{proof}

\begin{definition}
    The \textit{diagonal subalgebra} of $L_R\lspace$ is the  $R$-algebra given by
    \[D(L_R\lspace)=\vecspan_R\{s_{\alpha}p_As^*_{\alpha} \mid (\alpha, A, \alpha)\in E(S) \}.\]
\end{definition}

The following lemmas are standard and used throughout the paper.

\begin{lemma}\label{disjoint}
If $\{B_1,\ldots, B_n\}$ is a family of elements in an algebra of sets $\acf$, then there exists a family of pairwise disjoint sets $F=\{C_1,\ldots,C_m\}$ contained in $\acf$ such that $\bigcup B_j=\bigcup C_i$ and for every $j=1,\ldots,n$ there is a subset $E_j\scj F$ such that $B_j=\bigcup_{C\in E_j}C$.
\end{lemma}

\begin{lemma}\label{l:orthogonalhspan}
If $x\in \vecspan_R\{p_A \mid A\in\acf\}\scj L_R\lspace$ is written as
\[x=\sum_{i=1}^nr_ip_{A_i},\] for $A_i\in\acf$ and $r_i\in R$, $i=1,\ldots,n$, then there exists a family $\{C_1,\ldots,C_m\}$ of pairwise disjoint elements of $\acf$ and $s_1,\ldots,s_m\in R$ such that
\[x=\sum_{j=1}^ms_jp_{C_j},\]
and for all $j=1,\ldots,m$, there exists $i=1,\ldots,n$ such that $C_j\scj A_i$.
\end{lemma}

\begin{proof}
The proof follows immediately from Lemma~\ref{disjoint} and (i) of Definition~\ref{def:LPA}.
\end{proof}

\begin{definition}
A ring $R$ is \textit{$\mathbb{Z}$-graded} if there is a collection of additive subgroups $\{R_n\}_{n\in \mathbb{Z}}$ of $R$ such that 
    \begin{enumerate}
        \item $R=\bigoplus_{n\in \mathbb{Z}}R_n$, and 
        \item $R_mR_n \scj R_{m+n}$ for all $m,n\in \mathbb{Z}$.
    \end{enumerate}
The subgroup $R_n$ is called the \textit{homogeneous component of $R$ of degree n}. 
\end{definition}

\begin{definition}
If $R$ is a $\mathbb{Z}$-graded ring, then an ideal $I\scj R$ is a \textit{$\mathbb{Z}$-graded ideal} if $I=\bigoplus_{n\in \mathbb{Z}} (I\cap R_n)$. If $\phi:R\to S$ is a ring homomorphism between   $\mathbb{Z}$-graded rings, then $\phi$ is \textit{$\mathbb{Z}$-graded homomorphism} if $\phi(R_n)\scj  S_n$ for every $n\in \mathbb{Z}$.
\end{definition}

\begin{proposition}\label{prop:grading}
    The Leavitt labelled path algebra $L_R\lspace$ is $\zn$-graded, with grading given by
    \[L_R\lspace_n = \vecspan_R\{s_\alpha p_A s_\beta^* \mid \alpha,\beta\in\awstar,\ A\in\acfrg{\alpha}\cap\acfrg{\beta} \ \mbox{and} \ |\alpha|-|\beta|=n\}.\]
\end{proposition}

\begin{proof}

The proof is essentially the same as for Leavitt path algebras of graphs \cite[Corolary 2.1.5]{AbrAraMol}, and we therefore only provide an outline.

Let $D$ denote the free $R$-algebra generated by $\{p_A \mid A\in \acf \}\cup \{s_a,s_a^* \mid a\in\alf\}$. For  $a\in \alf$, define $l(s_a)=1$ and $l(s^*_a)=-1$. For $A\in \acf$ we define $l(p_A)=0$. For any monomial $rx_1\ldots x_n$, with $r\in R$ and $x_i\in \{p_A \mid A\in \acf \}\cup \{s_a,s_a^* \mid a\in\alf\}$, we let $l(rx_1\ldots x_n)=\sum_{i=1}^n l(x_i)$. Then $D$ has a natural $\mathbb{Z}$-grading by putting
\[D_n= \vecspan_R\{x_1\ldots x_n \mid x_i\in \{p_A \mid A\in \acf \}\cup \{s_a,s_a^* \mid a\in\alf\} \ \mbox{and} \ l(x_1\ldots x_n)=n \}\]
for $n\in\mathbb{Z}$.

Let $I\scj D$ be the ideal generated by relations (i)-(v) of Definition~\ref{def:LPA}.
Since $I$ is generated by homogeneous elements, it is a $\mathbb{Z}$-graded ideal. Hence, $L_R\lspace \cong D/I$ is $\mathbb{Z}$-graded, with the grading as stated above.
\end{proof}

\section{Partial skew group rings from labelled spaces}\label{section:partial.skew.group}

In this section, we recall from \cite{GillesDanie} the partial action associated with a labelled space $\lspace$ and show that there is a surjective homomorphism from $L_R\lspace$ to the partial skew ring arising from the aforementioned partial action.

Fix a weakly left-resolving labelled space $\lspace$. We describe the topology on $\ftight$, the tight spectrum of $\lspace$ as defined in Section~\ref{subsection:filters.E(S)}, inherited from the product topology on $\{0,1\}^{E(S)}$, where $\{0,1\}$ has the discrete topology. For each $e\in E(S)$, define
\begin{equation*} \label{dfn:tight.spec.basic.open.neigh}
	V_e=\{\xi\in\ftight\mid e\in\xi\}, 
\end{equation*}
and for $\{e_1,\ldots,e_n\}$ a finite (possibly empty) set in $E(S)$, we define
\[ V_{e:e_1,\ldots,e_n}= V_e\cap V_{e_1}^c\cap\cdots\cap V_{e_n}^c=\{\xi\in\ftight\mid e\in\xi,e_1\notin\xi,\ldots,e_n\notin\xi\}.\]
Sets of the form $V_{e:e_1,\ldots,e_n}$ form a basis of compact-open sets for a Hausdorff topology on $\ftight$ \cite[Remark 4.2]{Gil3}. Whenever we talk about a basic open set of $\ftight$, we mean a set of this form.

\begin{remark}\label{rmk:V_e.not.empty}
By the results of \cite[Section 12]{MR2419901}, $\ftight$ contains all ultrafilters of $E(S)$. In particular, an application of Zorn's lemma shows that for every $e\in E(S)\setminus\{0\}$, $V_e\neq\emptyset$.
\end{remark}

Before we recall the definition of the partial action associated with a labelled space, we present three lemmas that we will be needed during our work.

\begin{lemma}\label{lem:algebra.sets}
The algebra of sets over $\ftight$ generated by the family $\{V_e\mid e\in E(S)\}$ is equal to the set of compact-open subsets of $\ftight$.
\end{lemma}

\begin{proof}
Let $\alfg{A}$ be the algebra of sets over $\ftight$ generated by the family $\{V_e\mid e\in E(S)\}$ and $\alfg{K}$ the set of compact-open subsets of $\ftight$. Since $\alfg{K}$ is closed under finite unions, finite intersections and relative complements, and for each $e\in E(S)$, $V_e$ is compact-open, we have that $\alfg{A}\scj\alfg{K}$. On the other hand, sets of the form $V_{e:e_1,\ldots,e_n}$, for $e,e_1,\ldots,e_n\in E(S)$, form a basis of compact-open sets, so that each element $U\in \alfg{K}$ can be written as finite union of these sets. Clearly each $V_{e:e_1,\ldots,e_n}=V_e\setminus(V_{e_1}\cup\ldots\cup V_{e_n})$ is in $\alfg{A}$ and hence $\alfg{K}\scj\alfg{A}$.
\end{proof}

\begin{lemma}\label{lem:closed}
The set $V=\bigcup_{A\in\acf} V_{(\eword,A,\eword)}$ is closed in $\ftight$.
\end{lemma}

\begin{proof}
Suppose that $\{\xi^{(i)}\}_{i\in I}$ is a net in $V$ converging to $\xi\in\ftight$. Let $\alpha\in\awstar$ the word associated with $\xi$. If $\xi_0\neq\emptyset$, then for any $A\in\xi_0$, we have that $\xi\in V_{(\eword,A,\eword)}\scj V$. Supposing that $\xi_0=\emptyset$, we have that $\alpha\neq\eword$ and $r(A,\alpha_1)=\emptyset$ for all $A\in\acf$, but in this case, $V_{(\alpha_1,r(\alpha_1),\alpha_1)}$ is an open set containing $\xi$ and no element of the net $\{\xi^{(i)}\}_{i\in I}$, which is contradiction. 
\end{proof}

Fix a unital commutative ring $R$. For a topological space $X$, let $\Lc(X,R)$ denote the set of all locally constant functions $f:X\to R$ with compact support and let $1_V$, with $V\scj X$, denote the characteristic function on the set $V$.

\begin{lemma} \label{lem:DiagGenerators}
 $\Lc(\ftight, R)$ is generated as an $R$-algebra by the set
 \[\{1_{V_{e}} \mid e \in E(S) \}. \]
\end{lemma}
 
  \begin{proof}
   Suppose $f\in\Lc(\ftight, R)$. Since $f$ is locally constant with compact support, there is a finite collection of mutually disjoint compact-open sets $\{U_i\}_{i=1}^{n}$ in $\ftight$, and a finite set $\{r_1,\ldots,r_n\}\scj R$ such that $f=\displaystyle\sum_{i=1}^{n} r_{i}1_{U_{i}}$.
   This shows that the set of characteristic functions $1_U$, for $U$ compact-open subset of $\ftight$, generates $\Lc(\ftight,R)$. The result then follows from Lemma~\ref{lem:algebra.sets} and \cite[Lemma~2.2]{GDD2}.
 \end{proof}

Let $\mathbb{F}$ be the free group generated by $\alfg{A}$ (identifying the identity of $\mathbb{F}$ with $\eword$). For every $t\in \mathbb{F}$ there is a clopen set $V_t\scj \ftight$, which is compact if $t\neq\eword,$ and a homeomorphism $\varphi_t:V_{t^{-1}}\to V_t$ such that 
\begin{equation} \label{eq:PartialAction}
	\varphi=(\{V_t\}_{t\in\F},\{\varphi_t\}_{t\in\F})    
\end{equation}
is a topological partial action of $\F$ on $\ftight$ \cite[Proposition 3.12]{GillesDanie}. In particular, $V_\eword=\ftight$ and if 
$\alpha,\beta\in \awstar$, then $V_\alpha=V_{(\alpha,r(\alpha),\alpha)}, V_{\alpha^{-1}}=V_{(\eword,r(\alpha),\eword)}$, and $V_{(\alpha\beta^{-1})^{-1}}= \varphi_{\beta^{-1}}^{-1}(V_{\alpha^{-1}})$, with $V_{(\alpha\beta^{-1})^{-1}}\neq\emptyset$ if and only if $r(\alpha)\cap r(\beta)\neq\emptyset$ \cite[Lemma 3.10]{GillesDanie}. Moreover, if $V_t\neq \emptyset$, then $t=\alpha\beta^{-1}$ for some $\alpha,\beta\in \awstar$ \cite[Lemma 3.11(ii)]{GillesDanie}, in which case
\begin{equation}\label{eq:Vab-1}
   V_{\alpha\beta^{-1}}=V_{(\alpha,r(\alpha)\cap r(\beta),\alpha)} 
\end{equation}
by \cite[Lemma~4.3]{GDD2}. The definition of $\varphi_{\alpha\beta^{-1}}$ uses the functions defined in \cite[Section~4]{BoavaDeCastroMortari2}, however, we will not go through all the details here. For what is needed in this paper, it is sufficient to know that if $\xi\in V_{\beta\alpha^{-1}}$ with associated word $\beta\gamma$, then the associated word of $\varphi_{\alpha\beta^{-1}}(\xi)$ is $\alpha\gamma$ and, for $0\leq n\leq|\gamma|$, we have that
\begin{equation}\label{vaccinatedaligator}
    \varphi_{\alpha\beta^{-1}}(\xi)_{|\alpha|+n}=\{A\cap r(\alpha\gamma_{1,n})\mid A\in\usetr{\xi_{|\beta|+n}}{\acfrg{\gamma_{1,n}}}\},
\end{equation}
where $\usetr{\xi_{|\beta|+n}}{\acfrg{\gamma_{1,n}}}$ is the upper set of $\acfrg{\gamma_{1,n}}$ generated by $\xi_{|\beta|+n}$. In particular, if $\alpha=\eword$, we have that
\begin{equation}\label{eq:remove.beginning}
    \varphi_{\beta^{-1}}(\xi)_{n}=\usetr{\xi_{|\beta|+n}}{\acfrg{\gamma_{1,n}}}.
\end{equation}

The topological partial action $\varphi$ induces a dual algebraic partial action
$$\hat{\varphi}= (\{\Lc(V_t,R)\}_{t\in\mathbb{F}}, \{\hat{\varphi_t}\}_{t\in\mathbb{F}}),$$
where each $\hat{\varphi_t}:\Lc(V_{t^{-1}},R)\to \Lc(V_t,R)$ is the isomorphism defined by $$\hat{\varphi_t}(f)(\xi)=f\circ\varphi_{t^{-1}}(\xi).$$

 We let 
\[
\Lc(\ftight,R)\rtimes_{\hat{\varphi}} \mathbb{F}=\vecspan_R\left\{ \sum_{t\in\F}f_t\delta_t \mid f_t\in \Lc(V_t,R) \text{ and }   f_t\neq0 \text{ for finitely many }t\in\F \right\}, 
\]
be the partial skew group ring induced by $\hat{\phi}$. Note: $\delta_t$ has no meaning in itself and merely serves as a place holder. Multiplication in $\Lc(\ftight,R)\rtimes_{\hat{\varphi}} \mathbb{F}$ is given by 
\begin{equation*}
(a\delta_s) (b\delta_t)=\ph_s(\ph_{s^{-1}}(a)b)\delta_{st}.
\end{equation*}

Let $1_A$ denote the characteristic function on $V_{(\eword,A,\eword)}$, let $1_\alpha$ the characteristic function on $V_{(\alpha,r(\alpha),\alpha)}$ and let $1_{\alpha^{-1}}$ denote the characteristic function on $V_{(\eword,r(\alpha),\eword)}$.

\begin{proposition} \label{prop:SkewRingGenerators}
  The partial skew group ring $\Lc(\ftight,R)\rtimes_{\hat{\varphi}} \mathbb{F}$ over $R$ is generated by the set 
  \[\{1_{A}\delta_\eword, 1_a\delta_a, 1_{a^{-1}}\delta_{a^{-1}} \mid a\in \alf, A\in \acf \}.\]
\end{proposition}
  \begin{proof}
    Let $\mathcal{R}$ be the $R$-subalgebra of $\Lc(\ftight,R)\rtimes_{\hat{\varphi}} \mathbb{F}$ generated by the set \[\{1_{A}\delta_\eword, 1_a\delta_a,1_{a^{-1}}\delta_{a^{-1}} \mid a\in \alf, A\in \acf \}.\] It suffices to show that $\Lc(\ftight,R)\rtimes_{\hat{\varphi}} \mathbb{F} \scj \mathcal{R}$. Let $f_t\delta_t\in \Lc(\ftight,R)\rtimes_{\hat{\varphi}} \mathbb{F}$, with $f_t\in \Lc(V_{t},R)$ and $t\in\F$. Then $t=\alpha\beta^{-1}$ for some $\alpha,\beta\in \awstar$. If $t=\eword$ then, by \cite[Lemma 4.4]{GillesDanie}, we have that 
    \begin{equation} \label{eq:SkewRingGenerators}
        1_{V_{(\alpha,A,\alpha)}}\delta_\eword = (1_{V_{(\alpha,r(\alpha),\alpha)}}\delta_\alpha) (1_{V_{(\eword,A,\eword)}}\delta_\eword) (1_{V_{(\eword,r(\alpha),\eword)}}\delta_{\alpha^{-1}}).
    \end{equation}
    Thus, by Lemma~\ref{lem:DiagGenerators}, $f_\eword\delta_\eword\in\Lc(\ftight,R)\delta_\eword\scj  \mathcal{R}$.
    For $t\neq \eword$, we consider three cases. First assume that $|\alpha|,|\beta|\geq 1$. Then, by \cite[Lemma 4.4]{GillesDanie},  
	\begin{equation*}
		f_{\alpha\beta^{-1}}\delta_{\alpha\beta^{-1}} = (f_{\alpha\beta^{-1}}1_{\alpha\beta^{-1}})\delta_{\alpha\beta^{-1}} 
		= (f_{\alpha\beta^{-1}}\delta_\eword)(1_{\alpha}\delta_{\alpha})(1_{\beta^{-1}}\delta_{\beta^{-1}}) \in \mathcal{R}. 
	\end{equation*}
	Secondly, assume that $\beta=\eword$ and $\alpha\neq \eword$. Then 
	\begin{equation*}
		f_{\alpha}\delta_{\alpha} =  (f_{\alpha}1_{\alpha})\delta_{\alpha}
		= (f_{\alpha}\delta_\eword)(1_{\alpha}\delta_{\alpha}) \in \mathcal{R}. 
	\end{equation*}
	Lastly, assume that $\alpha=\eword$ and $\beta\neq \eword$. Then, by \cite[Lemma 4.4]{GillesDanie}, \begin{equation*}
		f_{\beta^{-1}}\delta_{\beta^{-1}} 
		= (f_{\beta^{-1}}1_{\beta^{-1}})\delta_{\beta^{-1}} 
		= (f_{\beta^{-1}}\delta_\eword)(1_{\beta^{-1}}\delta_{\beta^{-1}}) \in \mathcal{R}. 
	\end{equation*}
	Since $\Lc(\ftight,R)\rtimes_{\hat{\varphi}} \mathbb{F}$ is $R$-linearly spanned by elements of the form $f_t\delta_t$, it follows that $\Lc(\ftight,R)\rtimes_{\hat{\varphi}} \mathbb{F} \scj \mathcal{R}$, which completes the proof.
  \end{proof}
  
 The partial skew group ring $\Lc(\ftight,R)\rtimes_{\hat{\varphi}} \mathbb{F}$ has a natural $\mathbb{F}$-grading given by
\[\Lc(\ftight,R)\rtimes_{\hat{\varphi}} \mathbb{F} = \bigoplus_{t\in \mathbb{F}} \Lc(V_t,R)\delta_t,\]
\cite[Proposition 8.11]{MR3699795}. Using the map $\alpha\beta^{-1} \mapsto |\alpha|-|\beta|$, for $\alpha\beta^{-1}\in \mathbb{F}$ in reduced from, and the fact that the partial action $\varphi$ is semi-saturated \cite[Proposition 3.12]{GillesDanie}, it follows that $\Lc(\ftight,R)\rtimes_{\hat{\varphi}} \mathbb{F}$ has a $\mathbb{Z}$-grading with homogeneous component of degree $n\in \mathbb{Z}$ given by 
\[D_n=\vecspan_R\{f_{\alpha\beta^{-1}}\delta_{\alpha\beta^{-1}} \mid \alpha,\beta\in\awstar\ \mbox{and} \ |\alpha|-|\beta|=n \}. \]
  
\begin{proposition} \label{prop:OntoHomLeavittPartialSkew}
    There is a surjective $\mathbb{Z}$-graded homomorphism $\Phi: L_R\lspace \to \Lc(\ftight,R)\rtimes_{\hat{\varphi}} \mathbb{F}$ such that 
    \begin{equation*}
          \Phi(p_A) = 1_{A}\delta_\eword, \hspace{0.5cm}
          \Phi(s_a) = 1_a\delta_a, \hspace{0.5cm}
          \Phi(s^*_a) = 1_{a^{-1}}\delta_{a^{-1}}, 
    \end{equation*}
    for every $A\in \acf$ and $a\in \alf$.
\end{proposition}
  \begin{proof}
    Define $\Phi$ on the generators of $L_R\lspace$ as above. That $\Phi(p_A)$, $\Phi(s_a)$ and $\Phi(s^*_a)$ for $A\in\acf$ and $a\in\alf$ satisfy the defining relations of $L_R\lspace$ follows along the same line as is shown in the proof of \cite[Theorem~4.8]{GillesDanie} for labelled space $C^*$-algebras, and we therefore omit this. By the universality of $L_R\lspace$, $\Phi$ is well-defined and is an $R$-algebra homomorphism, whose image is the $R$-algebra generated by $\{1_{A}\delta_\eword, 1_a\delta_a, 1_{a^{-1}}\delta_{a^{-1}} \mid a\in \alf, A\in \acf \}$. Now, Proposition~\ref{prop:SkewRingGenerators} implies that $\Phi$ is surjective. That $\Phi$ is a $\mathbb{Z}$-graded homomorphism is immediate from the definitions of the gradings.
  \end{proof}
  
It is proved in Theorem~\ref{thm:SkewRingLeavittIsom} that the homomorphism $\Phi$ is an isomorphism. The proof is deferred to Section~\ref{section:graded.uniqueness.theorem}, as our proof relies on a Graded Uniqueness Theorem (Corollary~\ref{GUT}). 

We finish this section with the following lemma, which proves useful on various occasions in this paper.

\begin{lemma}\label{lem:projections} Let  $\lspace$ be a labelled space, and suppose $A\in\acf$ and $r\in R$. The following holds in $L_R\lspace$.
\begin{enumerate}[(i)]
    \item If $A\neq \emptyset$, then $p_A\neq 0$.
    \item If $rp_A=0$, then $r=0$ or $A=\emptyset$.
    \item If $\{A_1,\ldots, A_n\}$ are non-empty, pairwise disjoint elements of $\acf$ and $\sum r_i p_{A_i}=0$, then $r_i=0$ for all $i$.
    \item For paths $\alpha,\beta\in\awstar$, $A\in\acf$ such that $A\cap r(\alpha)\cap r(\beta)\neq\emptyset$, and a non-zero $r\in R$, we have that $rs_\alpha p_As^*_\beta\neq 0$.
\end{enumerate}

\end{lemma}
  \begin{proof}
   (i) If $A\neq \emptyset$, then $1_{A}\delta_\eword \in \Lc(\ftight,R)\rtimes_{\hat{\varphi}} \mathbb{F}$ is non-zero. Since $\Phi: L_R\lspace \to \Lc(\ftight,R)\rtimes_{\hat{\varphi}} \mathbb{F}$ is a homomorphism and $\Phi(p_A)=1_A\delta_\eword\neq 0$, it follows that $p_A\neq 0$ in $L_R\lspace$.

   (ii) Similarly to (i), if $r\neq 0$ and $A\neq \emptyset$, then $r1_{A}\delta_\eword \neq 0$ in $\Lc(\ftight,R)\rtimes_{\hat{\varphi}} \mathbb{F}$, which implies $rp_A \neq 0$  in $L_R\lspace$. 
   
   (iii) Since $\{A_1,\ldots, A_n\}$ are pairwise disjoint, we have that $0 = p_{A_i}(\sum r_j p_{A_j}) = r_ip_{A_i}$. Since $A_i\neq\emptyset$ for each $1\leq i\leq n$, it follows from (ii) that $r_i=0$ for for each $1\leq i\leq n$.
   
   (iv) If $\alpha,\beta\in\awstar$ and $A\in\acf$ are such that $A\cap r(\alpha)\cap r(\beta)\neq\emptyset$, then $V_{\alpha\beta^{-1}}\neq\emptyset$. 

   Thus, 
   \begin{eqnarray*}
      \Phi(rs_\alpha p_As^*_\beta) &=& \Phi(rs_\alpha p_{A\cap r(\alpha)\cap r(\beta)}s^*_\beta) \\
      &=& (r1_{\alpha}\delta_{\alpha})(1_{A\cap r(\alpha)\cap r(\beta)}\delta_\eword)(1_{\beta^{-1}}\delta_{\beta^{-1}}) \\
      &=& r1_{V_{(\alpha, A\cap r(\alpha)\cap r(\beta), \alpha)}}\delta_{\alpha\beta^{-1}} \\
      & \neq& 0,
   \end{eqnarray*}
   where the first equality follows from Proposition~\ref{prop:properties.of.c*-labelled.space}, and the last equality follows from \cite[Lemma 4.4]{GillesDanie}. 
   Hence $rs_\alpha p_As^*_\beta\neq 0$.
\end{proof}

\section{The Leavitt labelled path algebra as a Cuntz-Pimsner ring and the graded uniqueness theorem}\label{section:graded.uniqueness.theorem}

Our main goal in this section is to obtain the graded uniqueness theorem for Leavitt labelled path algebras. For this, we first use \cite[Theorem~3.1]{CFHL} to realize Leavitt labelled path algebras as Cuntz-Pimsner rings, and then apply it to the graded uniqueness theorem for Cuntz-Pimsner rings (\cite[Corollary~5.4]{CO55}).

For the reader's convenience, we reproduce \cite[Theorem~3.1]{CFHL} below, but due to the technical details required, we refer the reader to \cite{CO55,CFHL} for the definition of a $S$-system and its Cuntz-Pimsner ring. Before we proceed, we set up some notation necessary for \cite[Theorem~3.1]{CFHL}.

Let $S$ be a ring, $M$ a left $S$-module and $I$ a subset of $M$. The \emph{left annihilator} of $I$ by $S$, defined by \[\Ann_S(I)=\{r\in S \mid rx=0 \text{ for all } x\in I\},\]  is a left ideal of $S$. If $I$ is a submodule of $M$, then $\Ann_S(I)$ is a two-sided ideal of $S$. Let $J$ be a two-sided ideal of a ring $S$. We let $J^{\perp}$ denote the two-sided ideal \[J^{\perp}=\{r\in S \mid ry = yr = 0 \text{ for all } y\in J\}\] (for Leavitt path algebras of graphs these ideals are studied in \cite[Example~3.6]{CFHL}).

\begin{theorem}\label{hora}$($\cite[Theorem~3.1]{CFHL}$)$
Let $A=\bigoplus_{i\in \Z}A_i$ be a $\Z$-graded ring, $S$ a subring of $A_0$, and $I\scj A_1$ and  $J\scj A_{-1}$ additive subgroups such that 
\begin{enumerate}[\upshape(1)]
\item $SI, IS \scj I$, $SJ, JS \scj J$ and 
$JI\scj S$;

\medskip

\item For any finite subset $\{i_1,\dots,i_n\}\scj I$ there is an element $a$ in $IJ$ such that $a i_l=i_l$ for each $1\leq l \leq n$, and for 
any finite subset $\{j_1,\dots,j_m\}\scj J$ there is an element $b$ in $IJ$ such that $j_l b=j_l$ for each $1\leq l \leq m$;

\medskip

\item For $x\in \Ann_S(I)^\bot$ and $a\in IJ$, if $x-a \in \Ann_{A_0}(I)$, then $a\in S$;

\medskip 

\item $\Ann_S(I)\cap \Ann_S(I)^\bot=\{0\}$.
\end{enumerate}

Then there exists an $S$-bimodule homomorphism $\psi:J\otimes_S I \rightarrow S$ such that $\psi(j\otimes_S i)=ji$ for each $j\in J, i\in I$, and $(J,I,\psi)$ is an $S$-system. Furthermore, there is a graded isomorphism from the Cuntz-Pimsner ring $\mathcal O_{(J,I,\psi)}$ of the $S$-system $(J,I,\psi)$ to the subring of $A$ generated by $S,I,J$. 
\end{theorem}

\begin{remark}
In \cite{CFHL} the subring $S$ is denoted by $R$. We change the notation here to avoid confusion with the coefficient ring $R$ of $L_R\lspace$. 
\end{remark}

\begin{theorem}\label{LPACPR} Let $\lspace$ be a normal labelled space and $R$ a unital commutative ring. Then, there exists a subring $S\scj L_R\lspace_0 $ and an $S$-system $(J,I,\psi)$ such that $L_R\lspace$
is graded isomorphic to the Cuntz-Pimsner ring $\mathcal O_{(J,I,\psi)}$. 
\end{theorem}

\begin{proof}
In the notation of Theorem~\ref{hora}, let $S=\vecspan_R\{p_A \mid A\in \acf\}$ (which is a subring of $L_R\lspace$ and contained in the homogeneous component $L_R\lspace_0$), $I=\vecspan_R\{s_ap_A \mid A\in \acf, a\in \alf\}$ and $J=\vecspan_R\{p_As_a^* \mid A\in \acf, a\in \alf\}$. Next we verify that Conditions~(1) to (4) of Theorem~\ref{hora} are satisfied.

Condition (1) of Theorem~\ref{hora} follows directly from the relations defining $L_R\lspace$ (see Definition~\ref{def:LPA}). 

To see Condition (2) of Theorem~\ref{hora}, let $\{i_1,\ldots, i_n\}\scj I$. Notice that, for each $1\leq l \leq n$, we can write
\[ i_l = \sum_{j=1}^{n_l} \lambda_j^l s_{a_j^l} p_{A_j^l},
\]
where $A_j^l\in \acf_{a_j^l}$ and $\lambda^l_j\in R$. Let $F=\cup_{l=1}^{n}\{a_1^l, a_2^l, \ldots, a_{n_l}^l\}$ and put
\[a=\sum_{b\in F} s_{b} p_{r(b)} s_{b}^*,\]
which is an element of $IJ$. Using the relations in Definition~\ref{def:LPA}, it is easy to see that $a i_l = i_l$ for each $1\leq l \leq n$ as desired. The second part of Condition~(2) follows analogously.

Before proceeding to the proof of Conditions~(3) and (4), we need several claims.

{\bf Claim 1.}  We have that \[\Ann_S(I)=\vecspan_R\{p_A \mid A\scj\dgraph^0_{sink}\}.\]

Proof. Indeed, if $A\scj\dgraph^0_{sink}$ and $a\in\alf$, then by (i) and (ii) from Definition~\ref{def:LPA}, we have that $p_A(s_ap_B) = s_ap_{r(A,a)}p_B=s_ap_{\emptyset}p_B=0$ and then $\vecspan_R\{p_A:A\scj\dgraph^0_{sink}\}\scj \Ann_S(I)$. On the other hand, let $\sum_{j=1}^n\lambda_jp_{A_j}\in \Ann_S(I)$, where each $\lambda_j\neq0$ and the $A_j$'s are pairwise disjoint, and let $a\in\alf$. From the definition of $\Ann_S(I)$, for each $1\leq l\leq n$, we have that
\[0=\left(\sum_{j=1}^n\lambda_jp_{A_j}\right)(s_ap_{r(A_l,a)})= s_a\left(\sum_{j=1}^n\lambda_jp_{r(A_j,a)\cap r(A_l,a)}\right) = \lambda_ls_ap_{r(A_l,a)}.\]
By Items~(i) and (iv) in Lemma~\ref{lem:projections}, we have that $r(A_l,a)=\emptyset$. Since $a$ is arbitrary, we must have $A_l\scj \dgraph^0_{sink}$, which proves that $\sum_{j=1}^n\lambda_jp_{A_j}\in \vecspan_R\{p_A \mid A\scj\dgraph^0_{sink}\}$. $\lozenge$

{\bf Claim 2.} We have that \[\Ann_S(I)^{\perp}=\vecspan_R\{p_A \mid \nexists\,B\in\acf,\text{ such that }\emptyset \neq  B\scj A\cap \dgraph^0_{sink}\}.\]

Proof. Let $A\in\acf$ for which there does not exist $B\in\acf$ such that $\emptyset\neq B\scj A\cap \dgraph^0_{sink}$. To see that $p_A\in \Ann_S(I)^{\perp}$, let $C\in\acf$ such that $C\scj \dgraph^0_{sink}$. Since $p_Cp_A=p_Ap_C = p_{A\cap C}$, $A\cap C\in\acf$ and $A\cap C\scj A\cap \dgraph^0_{sink}$, we must have $A\cap C=\emptyset$ and hence $p_Ap_C=0$ as desired. For the other inclusion, let $\sum_{j=1}^n\lambda_jp_{A_j}\in \Ann_S(I)^{\perp}$, where $\lambda_j\neq0$ and $A_j$'s are pairwise disjoint. Let $1\leq l \leq n$ and $B\scj A_l\cap\dgraph^0_{sink}$. Then
\[0=\left(\sum_{j=1}^n\lambda_jp_{A_j}\right)p_B = \sum_{j=1}^n\lambda_jp_{A_j\cap B} = \lambda_lp_{A_l\cap B} =\lambda_lp_{B}.\]
By Lemma~4.5, we have that $B=\emptyset$. In other words, $p_{A_l}\in \{p_A \mid \nexists\, B\in\acf,\text{ such that }\emptyset \neq B\scj A\cap \dgraph^0_{sink}\}$ and, therefore, 
\[\Ann_S(I)^{\perp} \scj \vecspan_R\{p_A \mid \nexists \,B\in\acf,\text{ such that }\emptyset \neq B\scj A\cap \dgraph^0_{sink}\}.\ \lozenge\] 

{\bf Claim 3.}  If $a,b\in\alf$ are such that $a\neq b$, then $s_a^*xs_b=0$ for all $x\in S$.

Proof. This follows from (ii) and (iii) of Definition~\ref{def:LPA}. $\lozenge$

{\bf Claim 4.} If $a\in IJ$, then it can be written as
\[a=\sum_{j=1}^n \mu_js_{a_j}p_{B_j}s_{b_j}^*,\]
where $\mu_j\in R$,  $\mu_js_{a_j}p_{B_j}s_{b_j}^*\neq 0$, $B_j\in \acf_{a_j}\cap \acf_{b_j}$ for all $j$, and if $j\neq k$ are such that $a_j=a_k$ and $b_j=b_k$, then $B_j\cap B_k=\emptyset$.

Proof. Using (i) of Definition \ref{def:LPA}, we can write any $a\in IJ$ as a sum as above, not necessarily satisfying the extra conditions. By Proposition \ref{prop:properties.of.c*-labelled.space}(v) we may suppose, without loss of generality, that $B_j\in \acf_{a_j}\cap \acf_{b_j}$. Applying Lemma~\ref{disjoint}, and replacing the $B_j$'s by mutually disjoint $C_i$'s and summing the coefficients of the terms $s_{a_j}p_{C_i}s^*_{b_j}$ that appears multiple times, we get the claim. $\lozenge$

{\bf Claim 5.} Suppose that $x-a\in\Ann_{L_R\lspace_0}(I)$, where $x\in \Ann_S(I)^{\perp}$ and $a\in IJ$. Furthermore, suppose that $a$ is written as in Claim~4. Then $a_j=b_j$ for all $j$.

Proof. Fix $1\leq k\leq n$ and notice that $s_{a_k}^*(x-a)s_{b_k}=0$ because $x-a\in\Ann_{L_R\lspace_0}(I)$ and $s_{b_k}=s_{b_k}p_{r(b_k)}\in I$. Suppose, by contradiction, that $a_k\neq b_k$. Since $x\in S$, we have $x=\sum_i\lambda_ip_{A_i}$ and by Claim~3, $s_{a_k}^*xs_{b_k}=0$, from where we obtain $s_{a_k}^*as_{b_k}=0$. Using the characterization of $a\in IJ$ as in Claim~4, we have that
\begin{align*}
0 = s_{a_k}^*as_{b_k} & =  s_{a_k}^*\left(\sum_{j=1}^n \mu_js_{a_j}p_{B_j}s_{b_j}^*\right)s_{b_k} \\
& = \sum_{\substack{j \,:\, a_j=a_k \\ \text{ and } b_j=b_k}} \mu_jp_{r(a_k)}p_{B_j}p_{r(b_k)}\\
&= \sum_{\substack{j \,:\, a_j=a_k \\ \text{ and } b_j=b_k}} \mu_jp_{B_j}.
\end{align*}
By Lemma~\ref{lem:projections}(iii), we must have $\mu_jp_{B_j}=0$ for every $j$ such that $a_j=a_k$ and $b_j=b_k$, contradicting the condition $\mu_js_{a_j}p_{B_j}s_{b_j}^*\neq 0$ in Claim~4. Hence $a_j=b_j$ for all $j$. $\lozenge$

Let $a\in IJ$ and $x\in \Ann_S(I)^{\perp}$ be such that $x-a\in\Ann_{L_R\lspace_0}(I)$. As seen above in Claims~4 and 5, we may write $x=\sum_{i=1}^{m}\lambda_ip_{A_i}$, where $A_i$ are pairwise disjoint and $\lambda_i\neq 0$ for all $i$, and 
\[a=\sum_{j=1}^n \mu_js_{a_j}p_{B_j}s_{a_j}^*.\]
For the remainder of the claims, we retain these assumptions and fix these expressions for $a$ and $x$.

{\bf Claim 6}. 
For each $1\leq l\leq m$ and $1\leq k\leq n$ such that $r(A_l,a_k)\neq\emptyset$, we have that
\begin{equation}\label{eq:claim6}
\lambda_lp_{r(A_l,a_k)}=\sum_{\substack{j \,:\, a_j=a_k \text{ and } \\ r(A_l,a_k)\cap B_j\neq \emptyset}}\mu_jp_{r(A_l,a_k)\cap B_j}.
\end{equation}

Proof. Let $l$ and $k$ be as in the statement. Since $x-a\in\Ann_{L_R\lspace_0}(I)$ and $s_{a_k}\in I$, we have  $s_{a_k}^*(x-a)s_{a_k}p_{r(A_l,a_k)}=0$ and, hence, \[s_{a_k}^*xs_{a_k}p_{r(A_l,a_k)}=s_{a_k}^*as_{a_k}p_{r(A_l,a_k)}.\]
By replacing $x$ by $\sum_{i=1}^{m}\lambda_ip_{A_i}$, the left hand side of the above equation becomes
\begin{align*}
s_{a_k}^*xs_{a_k}p_{r(A_l,a_k)} &= s_{a_k}^*\left(\sum_{i=1}^{m}\lambda_ip_{A_i}\right)s_{a_k}p_{r(A_l,a_k)}\\
&= \sum_{i=1}^{m}\lambda_is_{a_k}^*s_{a_k}p_{r(A_i,a_k)}p_{r(A_l,a_k)}\\
&= \lambda_lp_{r(a_k)}p_{r(A_l,a_k)}\\
&= \lambda_lp_{r(A_l,a_k)},
\end{align*}
which is the left hand side of Equation~\eqref{eq:claim6}. Now, let's show that $s_{a_k}^*as_{a_k}p_{r(A_l,a_k)}$ is equal to the right hand side of Equation~\eqref{eq:claim6}. Indeed,
\begin{align*}
s_{a_k}^*as_{a_k}p_{r(A_l,a_k)} &= s_{a_k}^*\left(\sum_{j=1}^n \mu_js_{a_j}p_{B_j}s_{a_j}^*\right)s_{a_k}p_{r(A_l,a_k)} \\
&= \sum_{j=1}^n \mu_js_{a_k}^*s_{a_j}p_{B_j}s_{a_j}^*s_{a_k}p_{r(A_l,a_k)} \\
&= \sum_{j \,:\, a_j=a_k}  \mu_jp_{r(a_k)}p_{B_j}p_{r(a_k)}p_{r(A_l,a_k)} \\
&= \sum_{j \,:\, a_j=a_k}  \mu_jp_{B_j}p_{r(A_l,a_k)} \\
&= \sum_{\substack{j \,:\, a_j=a_k \text{ and } \\ r(A_l,a_k)\cap B_j\neq \emptyset}}  \mu_jp_{r(A_l,a_k)\cap B_j},
\end{align*}
as desired. $\lozenge$

{\bf Claim 7.} 
For each $1\leq k\leq n$, we have that
\[B_k=\bigcup_{i=1}^m r(A_i,a_k)\cap B_k.\]

Proof. We will show that $C_k=B_k\setminus\left(\bigcup_{i=1}^m r(A_i,a_k)\right)=\emptyset$. Notice that
\begin{align*}
    xs_{a_k}p_{C_k} &= \sum_{i=1}^m\lambda_ip_{A_i}s_{a_k}p_{C_k}\\
    &= \sum_{i=1}^m\lambda_is_{a_k}p_{r(A_i,a_k)}p_{C_k}\\
    &= \sum_{i=1}^m\lambda_is_{a_k}p_{r(A_i,a_k)\cap C_k}\\
    &= 0.
\end{align*}
Then
\begin{align*}
0 &= s^*_{a_k}(x-a)s_{a_k}p_{C_k}\\
&= -s^*_{a_k}as_{a_k}p_{C_k} \\
&= -s^*_{a_k}\left(\sum_{j=1}^n \mu_js_{a_j}p_{B_j}s_{a_j}^*\right)s_{a_k}p_{C_k} \\
&= -\sum_{j \,:\, a_j=a_k} \mu_jp_{B_j}p_{C_k}\\
&= -\mu_kp_{C_k},
\end{align*}
where the last equality follows from $B_j\cap C_k \scj B_j\cap B_k=\emptyset$ if $a_j=a_k$ and $j\neq k$. By hypothesis, $\mu_k\neq 0$ and, since $C_k\scj B_k \in \acf_{a_k}$, by Lemma~\ref{lem:projections} we get that $C_k=\emptyset$. $\lozenge$

{\bf Claim 8.} 
For each $1\leq l\leq m$ and each $c\in\alf$ such that $r(A_l,c)\neq\emptyset$, there exists $1\leq k\leq n$ such that $a_k=c$. Furthermore, $A_l\in\acf_{reg}$.

Proof. As in the above claims, $s^*_cxs_cp_{r(A_l,c)}=s^*_cas_cp_{r(A_l,c)}$. Since $r(A_l,c)\neq\emptyset$, we have that $s^*_cxs_cp_{r(A_l,c)}\neq0$. Therefore, $s^*_cas_cp_{r(A_l,c)}\neq0$. Hence, there is $1\leq k\leq n$ such that $a_k=c$. This means that $\lbf(A_l\dgraph^1) = \{c\in\alf \mid r(A_l, c)\neq\emptyset\} \scj \{a_1,a_2,\ldots,a_n\}$, from which we obtain that $|\lbf(A_l\dgraph^1)|<\infty$. Now, since $x\in \Ann_S(I)^{\perp}$, we have $A_l\not\scj \dgraph^0_{sink}$, which says that there exists $c\in\alf$ such that $r(A_l, c)\neq\emptyset$. In other words, $|\lbf(A_l\dgraph^1)|>0$, which shows that $A_l\in\acf_{reg}$. $\lozenge$

{\bf Claim 9.} 
We have that $x=a$ and, in particular, Condition~(3) of Theorem~\ref{hora} holds.

Proof. By Claim~8, for each $1\leq i \leq m$, we have $A_i\in\acf_{reg}$. Using Definition~\ref{def:LPA}(v), we get
\[x = \sum_{i=1}^m\lambda_ip_{A_i} = \sum_{i=1}^m\lambda_i\left(\sum_{c\in\lbf(A_i\dgraph^1)} s_cp_{r(A_i,c)}s^*_c\right) = \sum_{i=1}^m\sum_{c\in\lbf(A_i\dgraph^1)}s_c\lambda_ip_{r(A_i,c)}s^*_c.\]
By Claim~8, for each $c$ in the above sum, there exists $1\leq k\leq n$ such that $a_k=c$. Now, using Claim~6 with $c$ and $i$ in place of $a_k$ and $l$, respectively, we obtain
\begin{align*}
x &= \displaystyle\sum_{i=1}^m\sum_{c\in\lbf(A_i\dgraph^1)}s_c\lambda_ip_{r(A_i,c)}s^*_c \\
  &= \displaystyle\sum_{i=1}^m\sum_{c\in\lbf(A_i\dgraph^1)}s_c\left(\sum_{\substack{j\,:\,a_j=c \text{ and } \\ r(A_i,c)\cap B_j\neq \emptyset}}\mu_jp_{r(A_i,c)\cap B_j}\right)s^*_c \\
  &= \displaystyle\sum_{i=1}^m\sum_{c\in\lbf(A_i\dgraph^1)} \ \sum_{\substack{j\,:\,a_j=c \text{ and } \\ r(A_i,c)\cap B_j\neq \emptyset}}\mu_js_{a_j}p_{r(A_i,a_j)\cap B_j}s_{a_j}^*,
\end{align*}
where in the last equality we just replaced $c$ with $a_j$. Since $p_{r(A_i,a_j)\cap B_j}=0$ if $r(A_i,a_j)\cap B_j=\emptyset$, we can remove the condition $r(A_i,c)\cap B_j\neq \emptyset$ in the above sum, obtaining
\begin{align*}
x &= \displaystyle\sum_{i=1}^m\sum_{c\in\lbf(A_i\dgraph^1)} \ \sum_{\substack{j\,:\,a_j=c \text{ and } \\ r(A_i,c)\cap B_j\neq \emptyset}}\mu_js_{a_j}p_{r(A_i,a_j)\cap B_j}s_{a_j}^* \\
  &= \displaystyle\sum_{i=1}^m\sum_{c\in\lbf(A_i\dgraph^1)} \ \sum_{j\,:\,a_j=c}\mu_js_{a_j}p_{r(A_i,a_j)\cap B_j}s_{a_j}^* \\
  &= \displaystyle\sum_{i=1}^m\sum_{j\,:\,a_j\in\lbf(A_i\dgraph^1)}\mu_js_{a_j}p_{r(A_i,a_j)\cap B_j}s_{a_j}^* \\
  &=\sum_{i=1}^m\sum_{j=1}^n\mu_js_{a_j}p_{r(A_i,a_j)\cap B_j}s_{a_j}^*,
 \end{align*}
where the last equality follows from the fact that if $a_j\notin \lbf(A_i\dgraph^1)$, then $r(A_i,a_j)=\emptyset$. On the other hand, using Claim~7, we have
\begin{align*}
    a  &= \sum_{j=1}^n \mu_js_{a_j}p_{B_j}s_{a_j}^* \\
    &= \sum_{j=1}^n \mu_js_{a_j}\left(\sum_{i=1}^m p_{r(A_i,a_j)\cap B_j}\right)s_{a_j}^*\\
    &=\sum_{i=1}^m\sum_{j=1}^n \mu_js_{a_j}p_{r(A_i,a_j)\cap B_j}s_{a_j}^*,
\end{align*}
which shows that $x=a$, proving Claim~9 and completing the proof of Condition~(3). $\lozenge$

To show Condition~(4), let $x\in \Ann_S(I)\cap\Ann_S(I)^{\perp}$. By Claim~1, $x=\sum_{j=1}^n\lambda_jp_{A_j}$, where each $A_j \scj\dgraph^0_{sink}$.  Lemmas~\ref{disjoint} and \ref{l:orthogonalhspan} imply that we may also suppose the sets $A_j$ are pairwise disjoint. Since $x\in\Ann_S(I)^{\perp}$, and each $p_{A_k}\in\Ann_S(I)$, we get
\[0 = p_{A_k}x = p_{A_k}\sum_{j=1}^n\lambda_jp_{A_j} = \lambda_kp_{A_k}.\]
Since $k$ is arbitrary, $x=\sum_{j=1}^n\lambda_jp_{A_j}=0$, proving Condition~(4).

Finally, notice that $S$ contains the projections $p_A$, $I$ contains each $s_a$ and $J$ contains each $s_a^*$, so that $L_R\lspace$ is the subring of itself generated by $S,I,J$. Hence, it now follows from Theorem~\ref{hora} that we have an $S$-system $(J,I,\psi)$ such that $L_R\lspace$ is graded isomorphic to the Cuntz-Pimsner ring $\mathcal O_{(J,I,\psi)}$.
\end{proof} 

Using Theorem~\ref{LPACPR} and \cite[Corollary~5.4]{CO55} we obtain the following  Graded Uniqueness Theorem for Leavitt labelled path algebras:  

\begin{corollary}\label{GUT}[Graded Uniqueness Theorem](cf. \cite[Corollary~5.4]{CO55}) Let $R$ be a ring and $\lspace$ be a normal labelled space. If $A$ is a $\mathbb Z$-graded ring and $\eta : L_R\lspace \rightarrow A$ is a graded ring homomorphism with $\eta(rp_A)\neq 0$ for all non-empty $A\in\acf$ and all non-zero $r\in R$, then $\eta$ is injective.
\end{corollary}

\begin{proof}
Let $S=\vecspan\{p_A:A\in \acf\}\scj L_R\lspace$ as in the proof of Theorem \ref{LPACPR}. In order to prove that $\eta$ is injective, due to \cite[Corollary~5.4]{CO55} and \cite[Theorem~3.1]{CFHL}, we have to prove that $\eta(x)\neq 0$ for all $x\in S\setminus\{0\}$. By Lemma~\ref{l:orthogonalhspan}, if $x\neq 0$, then there is a family $\{A_1,\ldots,A_n\}$ of pairwise disjoint non-empty elements of $\acf$ and $r_1,\ldots,r_n\in R\setminus\{0\}$ such that
\[x=\sum_{i=1}^n r_ip_{A_i}.\]
If $\eta(x)=0$, then
\[0=\eta(x)=\eta(x)\eta(p_{A_1})=\eta(xp_{A_1})=\eta(r_1p_{A_1}),\]
however $\eta(r_1p_{A_1})\neq 0$ by hypothesis, arriving at a contradiction.
\end{proof}

We can now prove that the homomorphism $\Phi$ of Proposition~\ref{prop:OntoHomLeavittPartialSkew} is an isomorphism.
\begin{theorem}\label{thm:SkewRingLeavittIsom} 
    The graded homomorphism $\Phi: L_R\lspace \to \Lc(\ftight,R)\rtimes_{\hat{\varphi}} \mathbb{F}$ is injective, and therefore a graded isomorphism.
\end{theorem}

\begin{proof}
Since the map $\Phi$ of Proposition \ref{prop:OntoHomLeavittPartialSkew} is graded and for $r\in R\setminus\{0\}$ and $A\in\acf\setminus\{\emptyset\}$, we have that $r1_A\delta_{\omega}\neq 0$, the result follows from Corollary~\ref{GUT}.
\end{proof}

\section{Characterization of the Leavitt labelled path algebra as a Steinberg algebra}\label{section:steinberg.algebra}

Let $R$ be a  unital commutative ring. If $G$ is a groupoid, then $G^{(0)}$ denotes the set of units of $G$. If $G$ is a  topological groupoid, then $G$ is  \textit{\'etale} if the range and source maps $r,s:G\to G^{(0)}$ are local homeomorphisms. If, in addition, $G$ has basis of compact-open sets, then $G$ is an \textit{ample} groupoid. The \textit{Steinberg algebra} $A_R(G)$ associated with an ample groupoid $G$ is the $R$-algebra spanned by characteristic functions $1_U$, with $U\scj G$ a compact-open bisection. Multiplication is given by convolution: 
\[f*g(\gamma) = \sum_{s(\alpha)=s(\gamma)} f(\gamma\alpha^{-1})g(\alpha).  \]

Let 
\begin{equation*} \label{dfn:trans.groupoid}
	\mathbb{F} \ltimes_{\varphi} \ftight=\{(t,\xi)\in \mathbb{F} \times \ftight \mid \xi\in V_{t}\}
\end{equation*}
denote the transformation groupoid associated with the partial action $\varphi$ (see  (\ref{eq:PartialAction})) as defined in  \cite{MR2045419}. Then $\mathbb{F} \ltimes_{\varphi} \ftight$ is an ample groupoid \cite[Proposition 5.4]{GillesDanie}. We identify the unit space of $\mathbb{F} \ltimes_{\varphi} \ftight$ with $\ftight$. 

The following theorem realizes a Leavitt labelled path algebra as a Steinberg algebra.
\begin{theorem} \label{thm:LeavittSteinbergIsom}
There is an isomorphism $\Psi: L_R\lspace \to  A_R(\mathbb{F} \ltimes_{\varphi} \ftight)$ such that 
    \begin{equation*}
          \Psi(p_A) = 1_{\{\eword\}\times V_{(\eword,A,\eword)}}, \hspace{0.5cm}
          \Psi(s_a) = 1_{\{a\}\times V_{(a,r(a),a)}}, \hspace{0.5cm}
          \Psi(s^*_a) = 1_{\{a^{-1}\}\times V_{(\eword,r(a),\eword)}}, 
    \end{equation*}
    for every $A\in \acf$ and $a\in \alf$.
\end{theorem}
\begin{proof}
    Since $\ftight$ has a basis of compact-open sets, it is totally disconnected. Then, by \cite[Theorem 3.2]{MR3743184} we have that $\Lc(\ftight,R)\rtimes_{\hat{\varphi}} \mathbb{F} \cong  A_R(\mathbb{F} \ltimes_{\varphi} \ftight)$, and by Theorem~\ref{thm:SkewRingLeavittIsom} we have that $ L_R\lspace \cong  \Lc(\ftight,R)\rtimes_{\hat{\varphi}} \mathbb{F}$. Hence, there exists an isomorphism $\Psi$ from $L_R\lspace$ onto  $A_R(\mathbb{F} \ltimes_{\varphi} \ftight)$.
    
    By composing the two isomorphisms referred to in the preceding paragraph, we see that $\Psi$ maps the generators of $L_R\lspace$ as described in the statement of this theorem. We omit this computation. 
\end{proof}

Although Theorem~\ref{thm:LeavittSteinbergIsom} gives an explicit isomorphism defined on the generators, it will be useful in what follows to know how $\Psi$ maps general elements of the form $s_\alpha p_A s^*_\beta$. A short computation shows that 
\begin{equation}\label{eq:IsomOnTriples}
    \Psi(s_\alpha p_A s^*_\beta) = 1_{\{\alpha\beta^{-1}\}\times V_{(\alpha, A\cap r(\alpha)\cap r(\beta),\alpha)}}.
\end{equation}

\begin{proposition}\label{prop:diag}
  The diagonal subalgebra $D(L_R\lspace)$ is isomorphic to $\Lc(\ftight,R)$.
\end{proposition}

\begin{proof}
Since $\Lc(\ftight, R)$ is isomorphic to the subalgebra $\Lc(\ftight,R)\delta_\eword$ of $\Lc(\ftight,R)\rtimes_{\hat{\varphi}} \mathbb{F}$, it suffices to show that $D(L_R\lspace)$ is isomorphic to $\Lc(\ftight,R)\delta_\eword$.

Let $\Phi: L_R\lspace \to \Lc(\ftight,R)\rtimes_{\hat{\varphi}} \mathbb{F}$ be the graded isomorphism of Theorem~\ref{thm:SkewRingLeavittIsom}. Then, for any generator $s_{\alpha}p_As^*_{\alpha}$ of $D(L_R\lspace)$ and $r\in R$, we have  
 \begin{eqnarray*}
      \Phi(rs_{\alpha}p_As^*_{\alpha}) 
      &=& (r1_{\alpha}\delta_{\alpha})(1_{A\cap r(\alpha)}\delta_\eword)(1_{\alpha^{-1}}\delta_{\alpha^{-1}}) \\
      &=& r1_{V_{(\alpha, A\cap r(\alpha), \alpha)}}\delta_{\alpha\alpha^{-1}} \\
      &=& r1_{V_{(\alpha, A, \alpha)}}\delta_{\eword}. 
   \end{eqnarray*}
Therefore, Lemma~\ref{lem:DiagGenerators} implies that $\Phi$ maps the generators of $D(L_R\lspace)$ onto the generators of $\Lc(\ftight,R)\delta_\eword$. Hence $D(L_R\lspace)$ is isomorphic to $\Lc(\ftight,R)\delta_\eword$.
\end{proof}

From \cite[Theorem~4.5]{MR3743184}, we get the following characterization of diagonal preserving isomorphism of Leavitt labelled path algebras (cf.  Corollary~\ref{orb1qnob}).

\begin{corollary}\label{orb2qnob}  Let $\lspacei{i}$, $i=1,2$, be two normal labelled spaces, $(\varphi_i, \ftight_i)_{\mathbb{F}_i}$, $i=1,2$, be the associated partial actions, and suppose that $R$ is an integral domain. Then the following are equivalent:
\begin{enumerate}
    \item There is an isomorphism $\Phi:\Lc(\ftight_1,R)\rtimes_{\hat{\varphi_1}} \mathbb{F}_1 \rightarrow \Lc(\ftight_2,R)\rtimes_{\hat{\varphi_2}} \mathbb{F}_2 $ such that $\Phi (\Lc(\ftight_1, R)) = \Lc(\ftight_2, R) $.
    \item There is an isormorphism $\gamma: L_R\lspacei{1} \rightarrow L_R\lspacei{2}$ preserving diagonals, that is, $\gamma(D(L_R\lspace)_1)=D(L_R\lspace)_2.$
    \item The groupoids $\mathbb{F}_1 \ltimes_{\varphi_1} \ftight_1$ and $ \mathbb{F}_2 \ltimes_{\varphi_2} \ftight_2$ are isomorphic as topological groupoids.
\end{enumerate}
\end{corollary}

 Contrary to what happens with Leavitt path algebras of graphs and ultragraphs, where the unit, if it exists, can be written only in terms of the projections associated with sets of vertices, we might need to use the elements $s_a$ to describe the unit, if it exists, of a Leavitt labelled path algebra, as we see in the next corollary and example.

\begin{corollary}\label{unital} Let $\lspace$ be a normal labelled space. The following are equivalent:
    \begin{enumerate}[(i)]
        \item\label{item:LPAunital} $L_R\lspace$ is unital,
        \item\label{item:ftight.compact} $\ftight$ is compact,
        \item\label{item:unit} $\acf$ has a top element $I$ and the set $\alfg{F}=\{a\in\alf \mid  r(a)\setminus r(I,a)\neq\emptyset\}$ is finite.
    \end{enumerate}
    In this case, with the notation of item \eqref{item:unit},
    \[p_I+\sum_{a\in \alfg{F}}{s_ap_{r(a)\setminus r(I,a)}s_a^*}\]
    is the unit of $L_R\lspace$.
\end{corollary}

\begin{proof}
The equivalence of items \eqref{item:LPAunital} and \eqref{item:ftight.compact} follows from \cite[Proposition 4.11]{BenGroupoid} and Theorem~\ref{thm:LeavittSteinbergIsom}.

We prove that \eqref{item:ftight.compact} implies \eqref{item:unit}. For that, suppose that $\ftight$ is compact. By Lemma~\ref{lem:closed}, $V=\bigcup_{A\in\acf}V_{(\eword,A,\eword)}$ is closed in $\ftight$ and therefore compact. Then there exist $A_1,\ldots,A_n$ such that $V=\bigcup_{i=1}^nV_{(\eword,A_i,\eword)}$. We claim that $I=\bigcup_{i=1}^nA_i$ is the top element of $\acf$. If there exists $B\in\acf$ with $I\subsetneq B$, then an ultrafilter $\xi$ containing $(\eword,B\setminus I,\eword)$ is in $V$ but not in $V_{(\eword,A_i,\eword)}$ for any $i$, which is a contradiction. Now, if $\alfg{F}=\{a\in\alf \mid r(a)\setminus r(I,a)\neq\emptyset\}$ is infinite, then we could write $\ftight$ as an infinite disjoint union of open sets as follows
\begin{equation}\label{eq:ftight.dec}
    \ftight=V_{(\eword,I,\eword)}\cup\bigcup_{a\in\alfg{F}}V_{(a,r(a)\setminus r(I,a),a)},
\end{equation}
which contradicts the compactness of $\ftight$.

For \eqref{item:unit} implies \eqref{item:ftight.compact}, suppose that $\acf$ has a top element $I$ and the set $\alfg{F}=\{a\in\alf \mid r(a)\setminus r(I,a)\neq\emptyset\}$ is finite. Then the decomposition of $\ftight$ in Equation~\eqref{eq:ftight.dec} still holds and is a finite union of compact sets, and hence $\ftight$ is compact.

As for the expression for the unit, by \cite[Proposition 4.11]{BenGroupoid}, it must correspond to the characteristic function of $\ftight$ via the correspondence given in Theorem~\ref{thm:LeavittSteinbergIsom}. From Equation~\eqref{eq:ftight.dec}, we get that
\[1_{\ftight}=1_{V_{(\eword,I,\eword)}}+\sum_{a\in\acfg{F}}1_{V_{(a,r(a)\setminus r(I,a),a)}}.\]
Then, by Theorem~\ref{thm:LeavittSteinbergIsom} and Equation~\eqref{eq:IsomOnTriples}, we conclude that
\[p_I+\sum_{a\in \alfg{F}}{s_ap_{r(a)\setminus r(I,a)}s_a^*}\]
is the unit of $L_R\lspace$.
\end{proof}

\begin{example}
Consider a graph with only one edge $e$ with $s(e)\neq r(e)$, and a label $a$. If we use $\acf=\{\emptyset,r(e)\}$, then $r(e)$ is the top element of $\acf$. For the only label, we have that $r(a)\setminus r(r(a),a)=\{r(e)\}\setminus\emptyset\neq\emptyset$. By Corollary~\ref{unital}, the unit is then $p_{r(e)}+s_ap_{r(e)}s_a^{*}=p_{r(e)}+s_as_a^*$. Also notice that the only non-trivial projection $p_{r(e)}$ is such that $p_{r(e)}s_a=0$, so the unit cannot be written using only the generating projections.
\end{example}

\begin{remark}
In the non-unital case, we can look at some weaker properties such as having enough idempotents or having local units (see \cite[Definition~1.2.10]{AbrAraMol}), the former implying the latter. It is shown in \cite[Lemma~1.2.12]{AbrAraMol} that every Leavitt path algebra has enough idempotents, however, this does not need to be the case for Leavitt labelled path algebras (e.g. $\Lc(\overline{\nn},\mathbb{Z})$, where $\overline{\nn}$ is the one-point compactification of the natural numbers as seen in Example~\ref{ex:commutative}). Nevertheless, every Leavitt labelled path algebra has local units by Theorem~\ref{thm:LeavittSteinbergIsom} and \cite[Proposition~2.1]{BenModule}.
\end{remark}

\section{Examples}\label{section:examples}

In the first few examples, we show that Leavitt path algebras of graphs and ultragraphs can be seen as Leavitt labelled path algebras, and we give some sufficient conditions on the labelled space for its algebra to be isomorphic with a usual Leavitt path algebra. In the second part, we prove that any torsion-free commutative algebra generated by idempotents can be seen as a Leavitt labelled path algebra, and in particular, we give an example of a Leavitt labelled path algebra which is not isomorphic to the Leavitt path algebra associated with any graph or ultragraph. Finally, we provide a partial converse for Theorem~\ref{thm:SkewRingLeavittIsom}, which shows that certain partial skew group rings are isomorphic to Leavitt labelled path algebras.

\subsection{Leavitt path algebras of graphs and ultragraphs}

\begin{example}
Leavitt path algebras of graphs.
\end{example}
As in \cite[Example~3.3.(i)]{BP1}, given a graph $\dgraph$ we can build a labelled space where $\alf=\dgraph^1$, $\lbf:\dgraph^1\to\alf$ is the identity function and $\acf$ is the set of all finite subsets of $\dgraph^0$. Then $\lspace$ is a normal labelled space.

Let $L_R(\dgraph)$ be the usual Leavitt path algebra of the graph with generators $\{v\mid v\in\dgraph^0\}$ and $\{e,e^*\mid e\in\dgraph^1\}$ (see \cite{AbrAraMol}). There is an isomorphism $\phi:L_R\lspace\to L_R(\dgraph)$ given by $\phi(p_A)=\sum_{v\in A}v$ and $\phi(s_e)=e$. Indeed, that $\phi$ is a homomorphism follows from the universal property of $L_R\lspace$, that it is injective follows from the graded uniqueness theorem (Corollary~\ref{GUT}), and that it is surjective follows from observing that its image contains all the generators of $L_R(\dgraph)$.

\begin{example}
Leavitt path algebras of ultragraphs.
\end{example}

Recall that an \emph{ultragraph} is a quadruple $\mathcal{G}=(G^0, \mathcal{G}^1, r,s)$ consisting of two countable sets $G^0, \mathcal{G}^1$, a map $s:\mathcal{G}^1 \to G^0$, and a map $r:\mathcal{G}^1 \to \powerset{G^0}\setminus \{\emptyset\}$, where $\powerset{G^0}$ is the power set of $G^0$.

As in \cite[Example~3.3.ii]{BP1} and \cite[Definition~2.7]{GDD2}, given an ultragraph $(G^0,\mathcal{G}^1,r,s)$, we can build an associated labelled space: let $\dgraph = \dgraph_{\mathcal{G}}$ be the graph such that  $\dgraph^0 = G^0$, $\dgraph^1 = \{ ( s(e) , w ) \mid e \in \mathcal{G}^1 , w \in r (e)
\}$ and the range $r': \dgraph^1 \to \dgraph^0$  and source $s': \dgraph^1 \to \dgraph^0$ maps are defined by $r' ( v , w ) = w$ and $s' ( v , w ) = v$. Set $\alf=\mathcal{G}^1$ and define $\lbf: \dgraph^1 \to \alf$ by $\lbf ((s(e),w)) = e$. Let $\mathcal{G}^0$ be the set of all generalized vertices, that is, the smallest family containing all vertices and all ranges of edges, and is closed under finite intersections, finite unions and relative complements. If we let $\acf=\mathcal{G}^0$, then $\lspace$ is a normal labelled space. It is then straightforward to check that $L_R\lspace$ is isomorphic to $L_R(\mathcal G)$, via an isomorphism that takes generators to generators. Since every Leavitt path algebra of a graph can be seen as an ultragraph Leavitt path algebra, this example includes the previous one.
 
Leavitt path algebras associated with ultragraphs were originally defined in \cite{imanfar2017leavitt}, where the set $\mathcal{G}^0$ is not assumed to be closed under relative complements. In some works, for example \cite{reduction,goncalves_royer_2019, Nam}, $\mathcal{G}^0$ is assumed to be closed under relative complements. However, it is shown in \cite[Proposition~5.2]{GDD2} that the ultragraph Leavitt path algebras constructed with and without $\mathcal{G}^0$ being closed under relative complements are isomorphic. Hence, there is no harm in assuming that $\mathcal{G}^0$ is closed under relative complements.

\begin{example}
Label-finite left-resolving labelled graphs.
\end{example}

We have seen in the previous examples that ultragraph Leavitt algebras (and graph Leavitt path algebras) can be seen as Leavitt labelled path algebras. In this example we address the converse, that is, if a Leavitt labelled path algebra can be seen as a graph Leavitt path algebra. In fact, we provide a sufficient condition to obtain a positive answer to the aforementioned question. Our result is inspired by its C*-algebraic counterpart, \cite[Theorem~6.6]{BP1}, but also improves on it, as we are able to drop the row-finite assumption. However, in general, the converse is not true as is shown in Example~\ref{ex:commutative}.

\begin{definition}
    Let $\lgraph$ be a labelled graph. We say that $\lgraph$ is \emph{left-resolving} if for every $v\in\dgraph^0$, we have that $\lbf|_{r^{-1}(v)}$ is injective. We say that $\lgraph$ is \emph{label-finite} if $|\lbf^{-1}(a)|<\infty$ for every $a\in \alf$.
\end{definition}

\begin{proposition}(cf. \cite[Theorem~6.6]{BP1})\label{LLPALPA}. Let $\lgraph$ be a left-resolving, label-finite, labelled graph, and let $\acf$ be the family of all finite subsets of $\dgraph^0$. Then $L_R\lspace$ is isomorphic to $L_R(\dgraph)$.
\end{proposition}
\begin{proof}

Let $\{e,e^*,v\}_{e\in \dgraph^1, v\in \dgraph^0}$ denote the set of generators of the Leavitt path algebra $L_R(\dgraph)$ and, for each $A\in \acf$ and $a\in \alf $, define \[T_A = \sum_{v\in A}v,\quad Q_a= \sum_{\lbf(e)=a}e.\] Notice that the sum defining $Q_a$ is finite because the labelled graph is label-finite. To obtain a homomorphism from $L_R\lspace$ to $L_R(\dgraph)$ that maps $p_A$ to $T_A$ and $s_a$ to $Q_a$, we have to check that $\{T_A,Q_a\}_{A\in\acf,a\in\alf}$ satisfies the relations defining $L_R\lspace$ (see Definition~\ref{def:LPA}). We only check the Cuntz-Krieger condition (Condition (v) of Definition~\ref{def:LPA}) and leave the others to the reader.

Let $A\in \acf_{reg}$. Notice that $A$ contains no sinks (since it is regular), is finite, and contains no infinite emitters (since the graph is label-finite and $A$ is regular). We want to show that \begin{equation}\label{ck}T_A=\sum_{a\in\lbf(A\dgraph^1)}Q_a T_{r(A,a)}Q_a^*.
\end{equation}

Notice that, since the graph is left-resolving, \[Q_a T_{r(A,a)} = \sum_{\lbf(e)=a}e  \sum_{v\in r(A,a)} v = \sum_{\substack{\lbf(e)=a \\ r(e)\in r(A,a)}} e,\] so the right hand side of Equation~\eqref{ck} is equal to
\[\sum_{a\in\lbf(A\dgraph^1)} \sum_{\substack{\lbf(e)=a \\ r(e)\in r(A,a)}} e r(e) \sum_{\lbf(e')=a}r(e')(e')^* = \sum_{a\in\lbf(A\dgraph^1)} \sum_{\substack{\lbf(e)=a \\ r(e)\in r(A,a)}} e e^*. \]
To finish notice that, since the set $A$ is finite and regular, we can re-arrange the terms in the double sum on the right hand side of the equation above as 
\[ \sum_{v\in A} \sum_{s(e)=v} e e^* = \sum_{v\in A} v = T_A.\]

Now, from the universal property of $L_R\lspace$, we obtain an homomorphism $\phi: L_R\lspace \rightarrow L_R(\dgraph)$ such that $ \phi(p_A)= T_A$ and $\phi(s_a)= Q_A$. It is injective by the  Graded Uniqueness Theorem (Theorem~\ref{GUT}). To see that it is surjective, notice that $\phi(p_{\{v\}}) =v$ for every vertex $v$ and, for any edge $e$,  $e= Q_{\lbf(e)}T_{\{r(e)\}}= \phi(s_{\lbf(e)}p_{r(e)}) $. 
\end{proof}

\begin{remark} Observe that instead of using the graded uniqueness theorem in the proof above, we could instead define an inverse of $\phi$, which would be the homomorphism from $L_R(\dgraph)$ to $L_R\lspace$ such that $v\mapsto p_{\{v\}}$ and $e\mapsto s_{\lbf(e)}p_{r(e)}$.
\end{remark}

\subsection{Commutative algebras generated by idempotents and partial actions on them}

\begin{example}\label{ex:commutative}
Torsion-free commutative algebras generated by idempotents.
\end{example}
Let $A$ be a commutative $R$-algebra generated by its idempotents elements $E(A)$. Suppose also that $A$ is torsion-free in the sense that if $r$ is a non-zero element of $R$ and $e$ is a non-zero idempotent of $A$, then $re\neq 0$.\footnote{This is weaker than the condition given in \cite{Keimel}, but it is enough for the proof of \cite[Théorème~1]{Keimel}. His notion would imply the ring is an integral domain. On the other hand,  our definition is stronger than asking that $ra\neq 0$ for every $r$ that is not a zero-divisor and $a\neq 0$, which is sometimes used as the definition of a torsion-free module.} 
If $A\neq 0$, then this in particular implies that the only idempotents of $R$ are $0$ and $1$, because if $r\in R$ is idempotent then $1-r$ is an idempotent and $r(1-r)a=0$ for all $a\in A$. By \cite[Théorème~1]{Keimel}, $A$ is isomorphic to $\Lc(X,R)$ where $X$ is the Stone dual of the Boolean algebra of idempotents of $A$. More precisely, $X$ is the set of all ultrafilters in $E(A)$ with a basis for the topology given by the sets of the form $Z_e=\{\ft\in X\mid e\in\ft\}$ for $e\in E(S)$. The aforementioned isomorphism identifies $1_{Z_e}$ with $e$. We prove that $A$ is isomorphic to a Leavitt labelled path algebra.

Consider the graph where $\dgraph^0=X$ and $\dgraph^1=\emptyset$. Since there are no edges, there is no need for labels. More precisely, $\lbf$ is the empty function. We let $\acf$ be the compact-open subsets of $X$, which is an accommodating family closed under relative complements. Because there are no edges, we have that $\ftight\cong X$. Also, because there are no edges, we have that $L_R\lspace=D(L_R\lspace)=L_R\lspace_0$. Then, by Proposition~\ref{prop:diag}, we have that  $L_R\lspace\cong \Lc(X,R)\cong A$.

Now, we can give a family of examples of labelled spaces whose associated algebras are not isomorphic to any Leavitt path algebras of either graphs or ultragraphs. We can focus on the latter since it includes the former. It is shown in \cite[Corollary~6.14]{GDD2} that if a Leavitt path algebra of an ultragraph is commutative, then it is a direct sum of copies of $R$ or $R[x,x^{-1}]$. Notice that if $A$ is generated by idempotents, then its only possible $\zn$-grading is $A=A_0$, so $A$ cannot have $R[x,x^{-1}]$ as direct summand. Hence, if a Leavitt path algebra of an ultragraph is commutative and generated by idempotents, then it is necessarily a direct sum of copies of $R$, say $\bigoplus_{i\in I}R$. The space $X$ as above for this algebra must be the index set $I$ with the discrete topology. Because $X$ can be recovered from $\Lc(X,R)$ via the idempotents, if $X$ is a Stone space that is not discrete (for example, the one-point compactification of $\nn$), then $\Lc(X,R)$ is a Leavitt labelled path algebra that is not isomorphic to any Leavitt path algebra of an ultragraph.

\begin{example}
Partial actions on torsion-free commutative algebras generated by idempotents.
\end{example}

For this example, we retain the notation introduced in Example~\ref{ex:commutative}.

Let $A$ and $B$ be torsion-free commutative algebras generated by idempotents and assume that they are both unital. If $\theta:A\to B$ is a unital homomorphism then, when restricted to the idempotents, we get a Boolean algebra homomorphism $\theta:E(A)\to E(B)$. If $X$ and $Y$ are the Stone duals of $E(A)$ and $E(B)$ respectively, then there is a continuous function $\widehat{\theta}:Y\to X$ defined by $\widehat{\theta}(\ft)=\theta^{-1}(\ft)$. If $A$ and $B$ are not unital, then we need to impose an extra hypothesis on $\theta$ (see \cite{MR2974110}), however, this will not be needed in what follows.

Now, remove the hypothesis that $A$ is unital and let $I$ be a unital ideal of $A$. Then the unit $e$ of $I$ must be an idempotent of $A$ and via the isomorphism of $A$ and $\Lc(X,R)$ as explained in Example~\ref{ex:commutative}, it can be identified with $1_{Z_e}$. Moreover, we have that $I\cong\Lc(Z_e,R)$.

Suppose now that $\tau=(\{D_t\}_{t\in\F},\{\tau_t\}_{t\in\F})$ is a partial action of a free group $\F$ on $A$ such that $D_t$ is a unital ideal for every $t\in\F\setminus\{\eword\}$. For each $t\in\F\setminus\{\eword\}$, let $e_t$ be the identity of $D_t$ and $U_t=Z_{e_t}$, which is a compact-open subset of $X$. As explained above, for each $t\in \F\setminus\{\eword\}$, we can define a continuous function $\rho_t:U_{t^{-1}}\to U_t$ as $\rho_t=\widehat{\tau_{t^{-1}}}$. We also define $\rho_{\eword}=Id_X$. Analogous to what is done in \cite[Chapter~11]{MR3699795} (see also \cite{MR3506079,DE1}), we see that  $\rho=(\{U_t\}_{t\in\F},\{\rho_t\}_{t\in\F})$ is a topological partial action of $\F$ on $X$. Conversely, given such $\rho$, we can find a partial action $\tau$ such that these two constructions are inverses of each other. We say that $\rho$ and $\tau$ are \emph{dual actions} (notice the same arguments above, leading to the definition of dual actions, carry through if the free group is replaced by a discrete group).

Dual actions share interesting properties. More precisely, we have that $\tau$ is semi-saturated, in the sense that $\tau_{st}=\tau_s\circ \tau_t$ for all $s,t\in\ft$ such that $|st|=|s|+|t|$, if and only if $\rho$ is semi-saturated. Also, $\tau$ is orthogonal, in the sense that $D_a\cap D_b=\{0\}$ for generators $a,b\in\F$, if and only if $\rho$ is orthogonal in the sense that $U_a\cap U_b=\emptyset$ for all generators $a,b\in\F$. 

Our next goal is to prove that for a semi-saturated and orthogonal partial action $\tau$ of a free group $\F$ acting on $A$ as above, the skew group ring $A\rtimes_{\tau}\F$ is isomorphic to a Leavitt labelled path algebra. For that, we first work in a topological partial action setting. We highlight that the partial action $\varphi$ recalled in Section~\ref{section:partial.skew.group} is orthogonal and semi-saturated, \cite[Proposition~3.12]{GillesDanie}.

\begin{theorem}\label{thm:partial.action}
Let $X$ be a Stone space, that is, a dual of a Boolean algebra, and let $\rho=(\{U_t\}_{t\in\F},\{\rho_t\}_{t\in\F})$ be a semi-saturated, orthogonal topological partial action of a free group $\F$ on $X$ such that $U_t$ is compact-open for all $t\in\F\setminus\{\eword\}$. Then there exists a labelled space $\lspace$ and a homeomorphism $f:X\to\ftight$, where $\ftight$ is the tight spectrum of the inverse semigroup associated with $\lspace$, such that $f$ is equivariant with respect to the actions $\rho$ and $\varphi$ given in Section \ref{section:partial.skew.group}. In particular $\F\ltimes_\rho X$ and $\F\ltimes_\varphi\ftight$ are isomorphic as topological groupoids.
\end{theorem}

\begin{proof}
We start by constructing a labelled space. Let $\alf$ be the set of generators of $\F$, $\dgraph^0=X$, $\dgraph^1=\{(a,x)\in\alf\times X\mid x\in U_a\}$, $s(a,x)=x$, $r(a,x)=\rho_{a^{-1}}(x)$, $\lbf(a,x)=a$, and $\acf$ the set of compact-open subsets of $X$. Then $\lgraph$ is a labelled graph. Notice that, because each $\rho_{a^{-1}}$ is bijective, $r(a,x)=r(a,y)$ implies that $x=y$, so that the labelled graph $\lgraph$ is left-resolving. Because the action is semi-saturated, we see that, in general, for $\alpha\in\awplus$, we have that $r(\alpha)=U_{\alpha^{-1}}$ so that $\acfra=\{A\in\acf\mid A\scj U_{\alpha^{-1}}\}$. Moreover, for $A\in\acf$, we have that $r(A,\alpha)=\rho_{\alpha^{-1}}(A\cap U_{\alpha})$, which implies that $\acf$ is an accommodating family for $\lgraph$ such that $\lspace$ is a normal labelled space.

Notice that, because the partial action $\rho$ is orthogonal, we have a partition of $X$ given by $X=\dgraph^0_{sink}\sqcup\left(\bigsqcup_{a\in\alf}U_a\right)$. We associate an element of $\awleinf$ with a given $x\in X$ as follows: there is either a unique letter $a_1\in \alf$ such that $x\in U_{a_1}$, or $x\in \dgraph^0_{sink}$. In the first case, the same dichotomy applies to $\rho_{a_1^{-1}}(x)$, and so either there is a unique letter $a_2\in\alf$ such that $\rho_{a_1^{-1}}(x)\in U_{a_2}$ or $\rho_{a_1^{-1}}(x)\in \dgraph^0_{sink}$. Notice that because the partial action is semi-saturated, $\rho_{a_2^{-1}}(\rho_{a_1^{-1}}(x))=\rho_{(a_1a_2)^{-1}}(x)$ and, in particular, $x\in U_{a_1a_2}$. This process might either stop and we get a finite word $a_1\cdots a_n$ (possibly the empty word), or it does not stop and we get an infinite word $a_1a_2\cdots$. This word is exactly the label of the path given by the finite or infinite sequence of edges $(a_1,x),(a_2,\rho_{a_1^{-1}}(x)),\ldots,(a_n,\rho_{(a_1\cdots a_{n-1})^{-1}}(x))(\ldots)$ and thus it belongs to $\awleinf$. Denote by $\alpha_x$ this word and observe that $x\in U_{a_{1,i}}$ for all $1\leq i\leq |\alpha_x|$ (as in the definition of a complete family in Section~\ref{subsection:filters.E(S)},  understanding that $i\leq |\alpha_x|$ means $i<\infty$ if $|\alpha_x|=\infty$).

Next we want to associate a complete family with $x$ and $\alpha_x$. For each $0\leq i\leq |\alpha_x|$, we then let $\ftg{F}_i=\{A\in\acfrg{a_{1,i}}\mid \rho_{a_{1,i}^{-1}}(x)\in A\}$, which is an ultrafilter in $\acfrg{a_{1,i}}$ because it is a principal filter of a singleton. Suppose that $0\leq i<i+1\leq |\alpha_x|$, then 
\begin{align*}
\{A\in \acfrg{a_{1,i}} \mid r(A,a_{i+1})\in\ftg{F}_{i+1}\} & =\{A\in \acfrg{a_{1,i}} \mid \rho_{a_{1,i+1}}^{-1}(x)\in\rho_{a_{i+1}}^{-1}(A\cap U_{a_{i+1}})\} \\
&=\{A\in \acfrg{a_{1,i}} \mid \rho_{a_{1,i}}^{-1}(x)\in(A\cap U_{a_{i}})\}\\
&=\ftg{F}_i.
\end{align*}
Hence $\{\ftg{F}_i\}_{i=0}^{|\alpha_x|}$ is a complete family of ultrafilters for $\alpha_x$, that essentially corresponds to the sequence $\{x,\rho_{a_1^{-1}}(x),\ldots,\rho_{a_{1,n-1}^{-1}}(x)(\ldots)\}$.

Let $\xi_x$ be the filter corresponding to the pair $(\alpha_x,\{\ftg{F}_i\}_{i=0}^{|\alpha_x|})$ as in Theorem~\ref{thm.filters.in.E(S)} and let us prove that $\xi_x$ is tight. If $|\alpha_x|=\infty$, by Theorem~\ref{thm:TightFiltersType}, $\xi_x$ is a tight filter, because $\ftg{F}_i$ is an ultrafilter for all $i$. Suppose now that $n=|\alpha_x|<\infty$, so that $\rho_{\alpha_x^{-1}}(x)\in\dgraph^0_{sink}$. Let $A\in\ftg{F}_n$ and suppose that $|\lbf(A\dgraph^1)|<\infty$ and let us show that there exists $\emptyset\neq B\scj A\cap \dgraph^0_{sink}$. Notice that, because $|\lbf(A\dgraph^1)|<\infty$, $A$ intersects $U_a$ for a finite amount of letters so that $B=A\setminus \bigcup_{a\in\alf}U_a\scj \dgraph^0_{sink}$ is again a compact-open subset of $X$. Moreover, $\rho_{\alpha_x^{-1}}(x)\in B$, so that $B\neq\emptyset$ and is the desired set. By (ii) of Theorem~\ref{thm:TightFiltersType}, we have that $\xi_x$ is tight.

We then have a map $f:X\to\ftight$ given by $f(x)=\xi_x$, which is injective, because $\ftg{F}_0$ is completely determined by $x$. Next, we prove that $f$ is actually a homeomorphism that is equivariant with respect to the partial action $\rho$ on $X$ and $\varphi$ on $\ftight$.

For the surjectivity of $f$, let $\xia\in\ftight$. We cannot have $|\alpha|\geq 1$ and $\xi_0=\emptyset$, because if $A\in\xi_1$, then $r(\rho_{\alpha_1}(A),\alpha_1)=A$ so that $\rho_{\alpha_1}(A)\in\xi_0$. Let $x_{\xi}\in X$ be the element corresponding to $\xi_0$ via Stone duality, and let us prove that $f(x_{\xi})=\xi$. Suppose first that $\alpha=\eword$. If $x_\xi\notin \dgraph^{0}_{sink}$, then there exists $a\in\alf$ such that $x_\xi\in U_a$, but this would imply that $U_a\in\xi_0$. However, $\lbf(U_a\dgraph^1)=\{a\}$ is finite and $U_a\cap \dgraph^0_{sink}=\emptyset$, contradicting the fact that $\xi$ is tight. Hence, in this case, $\alpha_{x_\xi}=\eword$ and $f(x_\xi)=\xi$. Suppose now that $1\leq |\alpha|$ and let us prove that, for $1\leq i\leq|\alpha|$,  $x_\xi\in U_{\alpha_{1,i}}$ and $\xi_i$ is the ultrafilter in $\acfrg{\alpha_{1,i}}$ corresponding to the point $\rho_{\alpha_{1,i}^{-1}}(x_\xi)$. For $i=1$, we must have $x_\xi\in U_{\alpha_{1}}$ because otherwise we would find $A\in\xi_0$ such that $r(A,\alpha_1)=\emptyset$, which does not exist. Supposing that $\xi_1$ corresponds to a point $y\in X$ different from $\rho_{\alpha_1^{-1}}(x_\xi)$, we would have an element $B\in\xi_1$ such that $y\in B$ but $\rho_{\alpha_1^{-1}}(x_\xi)\notin B$. Then $A=\rho_{\alpha_1}(B)\in\acf$ is such that $r(A,\alpha_1)=B$ so that $A\in\xi_0$, but $x_\xi\notin A$, which is a contradiction. Using induction and repeating the same argument we conclude that $\xi_i$ is the ultrafilter in $\acfrg{\alpha_{1,i}}$ corresponding to the point $\rho_{\alpha_{1,i}^{-1}}(x_\xi)$. Now, if $|\alpha|$ is finite, as for the case of the empty word, we have that $\varphi_{\alpha^{-1}}(x)\in\dgraph^0_{sink}$ and hence $\alpha_x=\alpha$ and $f(x_\xi)=\xi$. On the other hand, if $|\alpha|=\infty$, then $\varphi_{\alpha{1,i}^{-1}}(x)\notin\dgraph^0_{sink}$ for all $i$ and again we have $\alpha=\alpha_x$ and $f(x_\xi)=\xi$. Hence $f$ is surjective and therefore a bijection.

By the above construction, we see that if $A\in\acf$ is such that $A\scj U_a$ for some $a\in\alf$, then $f(A)=V_{(a,\rho_{a^{-1}}(A),a)}$, or more generally, if $A\in\acf$ and $\alpha\in\awplus$ are such that $A\scj U_{\alpha}$, then $f(A)=V_{(\alpha,\rho_{\alpha^{-1}}(A),\alpha)}$.  In general, for $A\in\acf$, we have that $f(A)=V_{(\eword,A,\eword)}$. On the other hand, if $(\alpha,A,\alpha)\in E(S)$, then $A\scj U_{\alpha^{-1}}$ and $A'=\rho_{\alpha}(A)$ is such that $f(A')=V_{(\alpha,A,\alpha)}$. This implies that $f$ sends a basis of $X$ to a basis of $\ftight$ and therefore $f$ is a homeomorphism.

To show that $f$ is equivariant, since $\rho$ and $\varphi$ are semi-saturated, it enough to prove that $f(\rho_{a^{-1}}(x))=\varphi_{a^{-1}}(f(x))$ for all $a\in\alf$ and all $x\in U_a$. For that, we just have to check that $\varphi_{a^{-1}}(f(x))_0$ is the ultrafilter in $\acf$ corresponding to $\rho_{a^{-1}}(x)$. We already know that $f(x)_1$ is the ultrafilter in $\acfrg{a}$ corresponding to $\rho_{a^{-1}}(x)$. Now
\[\varphi_{a^{-1}}(f(x))_0=\usetr{f(x)_1}{\acf}=\{A\in\acf\mid \exists\, B\scj U_{a^{-1}}\text{ s.t. }\rho_{a^{-1}}(x)\in B\text{ and }B\scj A\},\]
and if $C\in\acf$ is such that $\rho_{a^{-1}}(x)\in C$ then, taking $B=C\cap U_{a^{-1}}$, we see that $C\in \varphi_{a^{-1}}(f(x))_0$ and hence $\varphi_{a^{-1}}(f(x))_0$ is the ultrafilter corresponding to $\rho_{a^-1}(x)$.

Since $f$ is equivariant, it is straightforward to check that the map $F:\F\ltimes_\rho X\to \F\ltimes_\varphi\ftight$ given by $F(t,x)=(t,f(x))$ is an isomorphism of topological groupoids.
\end{proof}

\begin{corollary}\label{jacarevoador}
    Let $R$ be a commutative unital ring and $A$ be a torsion-free commutative $R$-algebra generated by its idempotents elements. Let also $\tau=(\{D_t\}_{t\in\F},\{\tau_t\}_{t\in\F})$ be a semi-saturated, orthogonal algebraic partial action of a free group $\F$ on $A$ such that $D_t$ is unital for every $t\in\F\setminus\{\eword\}$. Then, there exists a normal labelled space $\lspace$ such that $A\rtimes_{\tau}\F\cong L_R\lspace$.
\end{corollary}

\begin{proof}
As explained before Theorem~\ref{thm:partial.action}, we may assume without loss of generality that $A=\Lc(X,R)$, where $X$ is some Stone space and $\tau$ is the dual of a partial action $\rho=(\{U_t\}_{t\in\F},\{\rho_t\}_{t\in\F})$ on $X$, which is also semi-saturated and orthogonal. The hypothesis that $D_t$ is unital for every $t\in\F\setminus\{\eword\}$ implies that $U_t$ is compact-open for these $t$'s.

By Theorem~\ref{thm:partial.action}, $\F\ltimes_\rho X$ and $\F\ltimes_\varphi\ftight$ are isomorphic as topological groupoids. So, by \cite[Theorem 3.2]{MR3743184} and Theorem~\ref{thm:LeavittSteinbergIsom} we then conclude that $\Lc(X,R)\rtimes_{\tau}\F\cong L_R\lspace$ as $R$-algebras.
\end{proof}

\section{The general uniqueness theorem}\label{section:general.uniqueness.theorems}

Let $R$ be a commutative unital ring and $\lspace$ a normal labelled space. In this section, we define a subalgebra of $L_R\lspace$ which we call the core subalgebra. We prove that it is maximal abelian and also that an $R$-algebra homomorphism $\phi$ of $L_R\lspace$ is injective if and only if $\phi$ is injective on the core subalgebra. To prove these results, we use the partial action defined in Section~\ref{section:partial.skew.group} and we describe the isotropy bundle of the associated groupoid.

\begin{definition}\cite[Definition 9.5]{COP}
	Let $\lspace$ be a normal labelled space.
	\begin{enumerate}
		\item A pair $(\alpha,A)$ with $\alpha\in\awplus$ and $A\in\acfra$ is a \emph{cycle} if for every $B\in\acfra$ with $B\scj A$, we have that $r(B,\alpha)=B$.
		\item A cycle $(\alpha, A)$ has an \textit{exit} if there exist $0\leq k \leq |\alpha|$ and $\emptyset\neq B\in \acf$ such that $B\scj r(A,\alpha_{1,k})$ and $\lbf(B\dgraph^1)\neq \{\alpha_{k+1}\}$ (where $\alpha_{|\alpha|+1}=\alpha_1$).
		\item The labelled space $\lspace$ satisfies condition ($L_\acf$) if every cycle has an exit.
	\end{enumerate}
\end{definition}

To fix notation, $\text{Iso}(G)^0$ represents the interior of the isotropy bundle of a groupoid $G$. Our next goal is to describe $\text{Iso}(\F\ltimes_\varphi\ftight)^0$. This set will play a crucial role in a generalized uniqueness theorem for $L_R\lspace$ as well as our simplicity characterization of $L_R\lspace$.

\begin{remark}\label{rmk:iso}
From \cite[Remark~6.11]{GillesDanie}, in order for a $(t,\xi)\in \F\ltimes_\varphi\ftight$ with $t\neq\eword$ to be a element of the isotropy bundle (not necessarily in its interior) it is necessary that there exist $\beta\in\awstar$, $\gamma\in\awplus$ such that the associated labelled path of $\xi$ is $\beta\gamma^{\infty}$ and $t$ is either $\beta\gamma\beta^{-1}$ or $\beta\gamma^{-1}\beta^{-1}$.
\end{remark}
In the next results we relate $\text{Iso}(\F\ltimes_\varphi\ftight)^0$ to cycles without exits.

\begin{lemma}\label{l:isocycle}
Suppose that $(\gamma,C)$ is a cycle without exits. Then, for every $\xi\in V_{(\gamma,C,\gamma)}$, we have that $\xi\in V_{\gamma}\cap V_{\gamma^{-1}}$ and $\varphi_{\gamma}(\xi)=\xi=\varphi_{\gamma^{-1}}(\xi)$.
\end{lemma}

\begin{proof}
The first part of the statement follows from the inequalities $(\gamma,C,\gamma)\leq (\gamma,r(\gamma),\gamma)$ and $(\gamma,C,\gamma)\leq (\eword,C,\eword)\leq (\eword,r(\gamma),\eword)$, where second inequality we uses that $r(C,\gamma)=C$. Also notice that, because $(\gamma,C)$ is a cycle without exits, we have that $C\in\acf_{reg}$. Then, using Theorem~\ref{thm:TightFiltersType}, given $\xi\in V_{(\gamma,C,\gamma)}$, we conclude that the labelled path associated with $\xi$ must be $\gamma^{\infty}$ and, for every $k\in\nn^*$, $(\gamma^k,C,\gamma^k)\in\xi$.

Now, due to the definition of a complete family and Theorem~\ref{thm.filters.in.E(S)}, in order to prove that $\xi=\varphi_{\gamma^{-1}}(\xi)$, it is sufficient to prove that $\xi_{k|\gamma|}=\varphi_{\gamma^{-1}}(\xi)_{k|\gamma|}$ for every $k\in\nn^*$. On the one hand, we have
\[\xi_{k|\gamma|}=\{A\in\acfrg{\gamma^k}\mid r(A,\gamma)\in \xi_{(k+1)|\gamma|}\},\]
and on the other hand, by Equation~\eqref{eq:remove.beginning}, we have
\[\varphi_{\gamma^{-1}}(\xi)_{k|\gamma|}=\usetr{\xi_{(k+1)|\gamma|}}{\acfrg{\gamma^k}}.\]

Suppose first that $A\in \xi_{k|\gamma|}$. Then $A\cap C\in \xi_{k|\gamma|}$ and, since $(\gamma,C)$ has no exits, $A\cap C=r(A\cap C,\gamma)\in \xi_{(k+1)|\gamma|}$. Hence $A\in \varphi_{\gamma^{-1}}(\xi)_{k|\gamma|}$.

Suppose now that $A\in \varphi_{\gamma^{-1}}(\xi)_{k|\gamma|}$, that is, there exists $B\in \xi_{(k+1)|\gamma|}$ such that $B\scj A$. Because $C\in\xi_{(k+1)|\gamma|}$, we may assume that $B\scj C$. Then, $B=r(B,\gamma)\scj r(A,\gamma)$, and hence $r(A,\gamma)\in \xi_{(k+1)|\gamma|}$. It follows that $A\in \xi_{k|\gamma|}$, concluding the proof that $\xi=\varphi_{\gamma^{-1}}(\xi)$.

Finally, by applying $\varphi_{\gamma}$ to the above equality, we get $\varphi_{\gamma}(\xi)=\xi$.
\end{proof}

\begin{lemma}\label{l:intiso}
Let $\beta\in\awstar$ and $\gamma\in\awplus$ be such that $\beta\gamma\in\awplus$. If $(\gamma,C)$ is a cycle without exits and $C\scj r(\beta)$, then $\{\beta\gamma\beta^{-1}\}\times V_{(\beta\gamma,C,\beta\gamma)}\scj \text{Iso}(\F\ltimes_\varphi\ftight)^0$ and $\{\beta\gamma^{-1}\beta^{-1}\}\times V_{(\beta\gamma,C,\beta\gamma)}\scj \text{Iso}(\F\ltimes_\varphi\ftight)^0$.
\end{lemma}

\begin{proof}
Note that since $(C,\gamma)$ is a cycle, $C=r(C,\gamma)\scj r(\beta\gamma)$ and $C=r(C,\gamma^2)\scj r(\beta\gamma^2)$. Moreover, we have that $(\beta\gamma,C,\beta\gamma)\leq(\beta,r(\beta)\cap r(\beta\gamma),\beta)$ and $(\beta\gamma,C,\beta\gamma)\leq(\beta\gamma,r(\beta)\cap r(\beta\gamma),\beta\gamma)$ so that, by Equation~\eqref{eq:Vab-1},
$V_{(\beta\gamma,C,\beta\gamma)}\scj V_{\beta\gamma^{-1}\beta^{-1}}\cap V_{\beta\gamma\beta^{-1}}$.
Hence, $\{\beta\gamma\beta^{-1}\}\times V_{(\beta\gamma,C,\beta\gamma)}$ and $\{\beta\gamma^{-1}\beta^{-1}\}\times V_{(\beta\gamma,C,\beta\gamma)}$ are indeed open sets of $\F\ltimes_\varphi\ftight$. Also, because the cycle $(\gamma,C)$ has no exists, we have that $r(C,\gamma_{1,n})\in\acf_{reg}$ for all $0\leq n\leq |\gamma|$, and, using Theorem~\ref{thm:TightFiltersType}, we conclude that if $\xi\in V_{(\beta\gamma,C,\beta\gamma)}$, then its associated word must $\beta\gamma^{\infty}$.

We claim that if $\xi\in V_{(\beta\gamma,C,\beta\gamma)}$, then $\varphi_{\beta\gamma\beta^{-1}}(\xi)=\xi$ and $\varphi_{\beta\gamma^{-1}\beta^{-1}}(\xi)=\xi$. Indeed, $\xi\in V_{\beta}$, and if $\eta=\varphi_{\beta^{-1}}(\xi)$, then $(\gamma,C,\gamma)\in\eta$, and by Lemma~\ref{l:isocycle}, $\varphi_{\gamma}(\eta)=\eta=\varphi_{\gamma^{-1}}(\eta)$. Also, $\eta\in V_{\gamma^{-1}\beta^{-1}}\cap V_{\gamma\beta^{-1}}$ so that
\begin{align*}
    \varphi_{\beta\gamma\beta^{-1}}(\xi)&=\varphi_{\beta\gamma}(\varphi_{\beta^{-1}}(\xi)) \\
    & =\varphi_{\beta\gamma}(\eta) \\
    & =\varphi_{\beta}(\varphi_{\gamma}(\eta)) \\
    &=\varphi_{\beta}(\eta)\\
    &=\xi,
\end{align*}
and similarly $\varphi_{\beta\gamma^{-1}\beta^{-1}}(\xi)=\xi$.

Finally, for $\xi\in V_{(\beta\gamma,C,\beta\gamma)}$, we have that
\[s(\beta\gamma\beta^{-1},\xi)=\varphi_{\beta\gamma^{-1}\beta^{-1}}(\xi)=\xi=r(\beta\gamma\beta^{-1},\xi),\]
so that $\{\beta\gamma\beta^{-1}\}\times V_{(\beta\gamma,C,\beta\gamma)}\scj \text{Iso}(\F\ltimes_\varphi\ftight)^0$. Analogously, $\{\beta\gamma^{-1}\beta^{-1}\}\times V_{(\beta\gamma,C,\beta\gamma)}\scj \text{Iso}(\F\ltimes_\varphi\ftight)^0$.
\end{proof}

\begin{lemma}\label{lem:NeighbEqualAllN}
Let $(\gamma,C)$ be a cycle without exits and $\beta\in\awstar$ be such that $\beta\gamma\in\awstar$. Then, for all $n\in\nn$, $V_{(\beta,C,\beta)}=V_{(\beta\gamma^n,C,\beta\gamma^n)}$.
\end{lemma}

\begin{proof}
As in the proof of Lemma~\ref{l:isocycle}, because $(\gamma,C)$ is a cycle without exits, we see that for any element $\xi\in V_{(\beta\gamma^n,C,\beta\gamma^n)}$, the associated word of $\xi$ must be $\xi\gamma^{\infty}$. Now, since $(\gamma,C)$ is a cycle, for all $n\in\nn$, $C=r(C,\gamma^n)$. Using 
Theorem~\ref{thm:TightFiltersType} and an induction argument, we then see that for $\xi\in\ftight$, we have $C\in\xi_{|\beta|}$ if and only if $C\in\xi_{|\beta\gamma^n|}$. Hence $V_{(\beta,C,\beta)}=V_{(\beta\gamma^n,C,\beta\gamma^n)}$.
\end{proof}

\begin{proposition}\cite[Proposition 6.10]{GillesDanie}\label{prop:non.trivial.isotropy}
Let $\lspace$ be a normal labelled space. A tight filter $\xia$ has non-trivial isotropy if and only if there exist $\beta,\gamma\in\awplus$ such that $\alpha=\beta\gamma^{\infty}$ and for all $n\in\nn^*$ and all $A\in\xi_{|\beta\gamma^n|}$ we have that $A\cap r(A,\gamma)\neq\emptyset$. 
\end{proposition}

The proof of the above proposition also shows the following lemma.

\begin{lemma}\label{l:isointer}
Let $(t,\xi) \in \text{Iso}(\F\ltimes_\varphi\ftight)$ with $t=\beta\gamma\beta^{-1}$ or $t=\beta\gamma^{-1}\beta^{-1}$ where $\beta\gamma^{\infty}$ is the labelled path associated with $\xi$. Then for all $n\in\nn^*$ and all $A\in\xi_{|\beta\gamma^n|}$ we have that $A\cap r(A,\gamma)\neq\emptyset$.
\end{lemma}

\begin{lemma}\label{l:isogoback}
Let $(t,\xi) \in \text{Iso}(\F\ltimes_\varphi\ftight)$ with $t=\beta\gamma\beta^{-1}$ or $t=\beta\gamma^{-1}\beta^{-1}$, where $\beta\gamma^{\infty}$ is the labelled path associated with $\xi$. Then, for all $n\in\nn$ and all $C\in\xi_{|\beta\gamma^n|}$, we have that $C\cap r(\beta)\in\xi_{|\beta|}$.
\end{lemma}

\begin{proof}
Because $(\beta\gamma\beta^{-1})^{-1}=\beta\gamma^{-1}\beta^{-1}$, in both cases, we have that $\varphi_{\beta\gamma^{-1}\beta^{-1}}(\xi)=\xi$. Using Equation~\eqref{vaccinatedaligator}, we get that
\[\xi_{|\beta\gamma^n|}=\{B\cap r(\beta\gamma^n)\mid B\in \usetr{\xi_{|\beta\gamma^{n+1}|}}{\acfrg{\gamma^n}}\}.\]

A simple induction argument then shows that for all $n\in\nn$ and $C\in\xi_{|\beta\gamma^n|}$, we have that
\[C\cap r(\beta)\cap r(\beta\gamma)\cap\cdots\cap r(\beta\gamma^{n-1})\in\xi_{|\beta|},\]
and because $\xi_{|\beta|}$ is a filter, we conclude that $C\cap r(\beta)\in\xi_{|\beta|}$.
\end{proof}

\begin{lemma}\label{lem:ExistCycleNeighb}
Let $(t,\xi) \in \text{Iso}(\F\ltimes_\varphi\ftight)$ with $t=\beta\gamma\beta^{-1}$ or $t=\beta\gamma^{-1}\beta^{-1}$, where $\beta\gamma^{\infty}$ is the labelled path associated with $\xi$. 
Then, for every neighborhood $U\scj \ftight$ of $\xi$, there exists $k>1$  and  $B\in\xi_{|\beta\gamma^k|}$ such that
\[ \xi\in V_{(\beta\gamma^k,B,\beta\gamma^k)}\scj U. \]
\end{lemma}

\begin{proof}
We may assume that $U$ is of the form $V_{e:e_1,\ldots,e_n}$ for some $e=(\alpha,A,\alpha),e_1=(\alpha_1,A_1,\alpha_1),\ldots,e_n=(\alpha_n,A_n,\alpha_n)\in E(S)$. Let $k\in\nn$ be such that $k>1$ and $|\beta\gamma^k|\geq\max\{|\alpha|,|\alpha_1|,\ldots,|\alpha_n|\}$. By rearranging the indices, if necessary, we assume that for some $0\leq m\leq n$, we have that $\alpha_1,\ldots,\alpha_m$ are beginnings of $\beta\gamma^k$ and $\alpha_{m+1},\ldots,\alpha_{n}$ are not comparable with $\beta\gamma^k$. For $i=1,\ldots,m$, let $\delta_i$ be such that $\alpha_i\delta_i=\beta\gamma^k$, and let $\delta$ be such that $\alpha\delta=\beta\gamma^k$. Define $B=r(A,\delta)\setminus\left(\bigcup_{i=1}^m r(A_i,\delta_i)\right)$. Using Remark~\ref{remark.when.in.xialpha}, the order in $E(S)$ and that $\xi_{|\beta\gamma^k|}$ is an ultrafilter in a Boolean algebra, we see that $B\in\xi_{|\beta\gamma^k|}$, and in particular $B\neq\emptyset$. Moreover, since the $\alpha_{m+1},\ldots,\alpha_n$'s are not comparable with $\beta\gamma^k$, we have that 
\[\xi\in V_{(\beta\gamma^k,B,\beta\gamma^k)}\scj V_{e:e_1,\ldots,e_n}.\]
\end{proof}

\begin{lemma}\label{lem:InteriorMustBeCycle}
Let $(t,\xi) \in \text{Iso}(\F\ltimes_\varphi\ftight)$ with $t=\beta\gamma\beta^{-1}$ or $t=\beta\gamma^{-1}\beta^{-1}$, where $\beta\gamma^{\infty}$ is the labelled path associated with $\xi$.  Suppose that for every $C\in\xi_{|\beta|}$ one of the following two conditions holds:
\begin{enumerate}[(i)]
    \item $(\gamma,C)$ is a cycle with an exit,
    \item $(\gamma,C)$ is not a cycle,
\end{enumerate}
then $\xi\notin \text{Iso}(\F\ltimes_\varphi\ftight)^0$.
\end{lemma}

\begin{proof}
In order to prove that $\xi\notin \text{Iso}(\F\ltimes_\varphi\ftight)^0$, it is sufficient to show that if $U$ is an open neighborhood of $\xi$ such that $U\scj V_t$, then there exists $\eta\in U$ such that $(t,\eta)\notin \text{Iso}(\F\ltimes_\varphi\ftight)$. Also, we may assume that $U$ is of the form $V_{e:e_1,\ldots,e_n}$ for some $e=(\alpha,A,\alpha),e_1=(\alpha_1,A_1,\alpha_1),\ldots,e_n=(\alpha_n,A_n,\alpha_n)\in E(S)$. 
By Lemma~\ref{lem:ExistCycleNeighb}, there is a $k>1$ and a $B\in\xi_{|\beta\gamma^k|}$ such that 
\[\xi\in V_{(\beta\gamma^k,B,\beta\gamma^k)}\scj V_{e:e_1,\ldots,e_n}.\]

Also by Lemma~\ref{l:isogoback}, $C=B\cap r(\beta)\in\xi_{|\beta|}$. Suppose first that $(\gamma,C)$ is a cycle so that $(\beta\gamma^k,C,\beta\gamma^k)\in E(S)$. By assumption $(\gamma,C)$ it has an exit, that is, there exists $0\leq l \leq|\gamma|$ and $\emptyset\neq D\scj r(C,\gamma_{1,l})$ such that $\lbf(D\dgraph^1)\neq \{\gamma_{|l+1|}\}$. If $\lbf(D\dgraph^1)=\emptyset$, we consider $f=(\beta\gamma^k\gamma_{1,l},D,\beta\gamma^k\gamma_{1,l})$, and if $a\in \lbf(D\dgraph^1)\neq \{\gamma_{|l+1|}\}$, we let $f=(\beta\gamma^k\gamma_{1,l}a,r(D,a),\beta\gamma^k\gamma_{1,l}a)$. In either case, if we take $\eta$ such that $f\in\eta$ (see Remark~\ref{rmk:V_e.not.empty} for the existence of such $\eta$), then $\eta\in V_{(\beta\gamma^k,B,\beta\gamma^k)}$, but $(t,\eta)\notin \text{Iso}(\F\ltimes_\varphi\ftight)$, because the associated word of $\eta$ is either $\beta\gamma^k\gamma_{1,l}$ in the first case, or starts with $\beta\gamma^k\gamma_{1,l}a$ in the second case.

Suppose now that $(\gamma,C)$ is not a cycle, then neither is $(\gamma,B)$. Then, either $B\setminus r(B,\gamma)\neq\emptyset$ or $r(B,\gamma)\setminus B\neq\emptyset$. In the first case, if we let $f=(\beta\gamma^k,B\setminus r(B,\gamma),\beta\gamma^k)$ and $\eta$ such that $f\in \eta$, we have that $(t,\eta)\notin \text{Iso}(\F\ltimes_\varphi\ftight)$, because of Lemma~\ref{l:isointer} and of the equality
\[(B\setminus r(B,\gamma))\cap r(B\setminus r(B,\gamma),\gamma)=(B\setminus r(B,\gamma))\cap (r(B,\gamma)\setminus r(B,\gamma^2))=\emptyset.\]
The second case is analogous with $f=(\beta\gamma^{k+1},r(B,\gamma)\setminus B,\beta\gamma^{k+1})$.
\end{proof}

\begin{proposition}\label{p:isobasis}
The collection of sets of the form $\{t\}\times V_{(\beta,C,\beta)}$, where $t=\beta\gamma\beta^{-1}$ or $t=\beta\gamma^{-1}\beta^{-1}$ for some cycle $(\gamma,C)$ without exists  and such that $\beta\gamma\in\awstar$, together with the empty set, forms a basis of compact-open bisections for $\text{Iso}(\F\ltimes_\varphi\ftight)^0 \setminus (\F\ltimes_\varphi\ftight)^{(0)}$. Moreover, this basis is closed under finite intersections.
\end{proposition}

\begin{proof}
We start by observing that if $(\gamma,C)$ is a cycle, then $\gamma\in\awplus$, so that for every  $\beta\in \awstar$, we have that $\beta\gamma\beta^{-1}\neq\eword$ and $\beta\gamma^{-1}\beta^{-1}\neq\eword$ in $\F$. 

Suppose $(t,\xi)\in \text{Iso}(\F\ltimes_\varphi\ftight)^0 \setminus (\F\ltimes_\varphi\ftight)^{(0)}$. Then, by Remark~\ref{rmk:iso} and Lemma~\ref{lem:InteriorMustBeCycle}, $ t=\beta\gamma\beta^{-1}$ or $t=\beta\gamma^{-1}\beta^{-1}$ and there exists $C\in \xi_{|\beta|}$ such that $(\gamma, C)$ is a cycle without exists. 

Let $U\scj \F\ltimes_\varphi\ftight$ be a compact-open bisection containing $(t,\xi)$. Since $\F\ltimes_\varphi\ftight$ endowed with the product topology is a regular space and $(\F\ltimes_\varphi\ftight)^{(0)}$ is closed, we may assume that $U\scj \text{Iso}(\F\ltimes_\varphi\ftight)^0 \setminus (\F\ltimes_\varphi\ftight)^{(0)}$. Let $V=r(U)\scj \ftight$. Then $V$ is a compact-open neighborhood of $\xi$ in $\ftight$. By Lemma~\ref{lem:ExistCycleNeighb} there there is a $k\in\N$ and $B\in\xi_{|\beta\gamma^k|}$ such that 
\begin{equation*}
    \xi\in V_{(\beta\gamma^k,B,\beta\gamma^k)}\scj V.
\end{equation*}

Since $(\gamma, C)$ is a cycle without exits, it follows from Lemma~\ref{lem:NeighbEqualAllN}, that $V_{(\beta,C,\beta)}=V_{(\beta\gamma^m, C, \beta\gamma^m)}$ for every $m\in \N$. 
Now, since $\xi_{|\beta\gamma^k|}$ is a filter, we have that $\emptyset\neq B\cap C \in \xi_{\beta\gamma^k}$, and thus, by Lemma~\ref{l:isogoback}, we have that $(B\cap C) \cap r(\beta)\in \xi_{|\beta|}$. Therefore, $(\gamma, B\cap C)$ is a cycle without exits such that 
\begin{equation}\label{eq:IsoIntBasis}
    \xi\in V_{(\beta\gamma^k,B\cap C,\beta\gamma^k)}\scj V_{e:e_1,\ldots,e_n}.
\end{equation}
In fact, (\ref{eq:IsoIntBasis}) holds for all $k\in \N^*$, by Lemma~\ref{lem:NeighbEqualAllN}. 
Thus, $\{t\}\times V_{(\beta,C,\beta)}$ is a compact-open neighborhood of $(t,\xi)$ such that $\{t\}\times V_{(\beta,C,\beta)} \scj U$. This proves that sets of the form $\{t\}\times V_{(\beta,C,\beta)}$ which have the properties described in the statement form a basis of compact-open sets for $\text{Iso}(\F\ltimes_\varphi\ftight)^0 \setminus (\F\ltimes_\varphi\ftight)^{(0)}$. 

Next we show that this basis is closed under finite intersections. Suppose $\{s\}\times V_{(\alpha,B,\alpha)}$ and $\{t\}\times V_{(\beta,C,\beta)}$ are basic compact-open sets in $\text{Iso}(\F\ltimes_\varphi\ftight)^0 \setminus (\F\ltimes_\varphi\ftight)^{(0)}$ with $(\delta, B)$ and $(\gamma,C)$ cycles without exits. If $(\{s\}\times V_{(\alpha,B,\alpha)}) \cap (\{t\}\times V_{(\beta,C,\beta)}) = \emptyset$, then we are done. Suppose now that 
$(\{s\}\times V_{(\alpha,B,\alpha)}) \cap (\{t\}\times V_{(\beta,C,\beta)}) \neq \emptyset$. Then $s=t$ and $V_{(\alpha,B,\alpha)}\cap V_{(\beta,C,\beta)} \neq \emptyset$. Since $(\delta, B)$ and $(\gamma,C)$ are cycles without exits, Lemma~\ref{lem:NeighbEqualAllN} implies that  
\[\emptyset\neq V_{(\alpha,B,\alpha)}\cap V_{(\beta,C,\beta)}= V_{(\alpha\delta^m,B,\alpha\delta^m)}\cap V_{(\beta\gamma^n,C,\beta\gamma^n)}\]
for all $m,n\in \N^*$, which in turn implies that $\alpha\delta^\infty=\beta\gamma^\infty$. 
In particular, $\alpha$ is a beginning of $\beta$ or vice-versa. Assume the former, that is, $\beta=\alpha\beta'$ for some $\beta'\in\awstar$. Then
\[(\{t\}\times V_{(\alpha,B,\alpha)})\cap (\{t\}\times V_{(\beta,C,\beta)})=\{t\}\times V_{(\beta,r(B,\beta')\cap C,\beta)}.\]
Now it follows from Lemma \ref{l:intiso} that $\{t\}\times V_{(\beta,r(B,\beta')\cap C,\beta)}$ is a basic compact-open neighborhood in $\text{Iso}(\F\ltimes_\varphi\ftight)^0 \setminus (\F\ltimes_\varphi\ftight)^{(0)}$ because $r(B,\beta')\cap C\scj C$ and therefore $(r(B,\beta')\cap C,\gamma)$ is a cycle without exits.

Finally, because $\F\ltimes_\varphi\ftight$ is ample there is no harm in assuming that $U$ in the second paragraph of this proof is a bisection as well. Then (\ref{eq:IsoIntBasis}) implies that the basis for $\text{Iso}(\F\ltimes_\varphi\ftight)^0 \setminus (\F\ltimes_\varphi\ftight)^{(0)}$ consists of compact-open bisections.  
\end{proof}

For ultragraphs, it is proved in \cite{GDD2} that if $(t,\xi)\in \text{Iso}(\F\ltimes_\varphi\ftight)^{0}$ and $t\neq\eword$, then $(t,\xi)$ is an isolated point. That is, for ultragraph (and consequently, graph) groupoids, the elements that are in the interior of the isotropy and which are not in the unit space are isolated points. This needs to be the case for labelled spaces, as illustrated in the next example.

\begin{example}
Let $\dgraph^0=\nn\cup\{\infty\}$, $\dgraph^1=\{e_x\mid x\in\dgraph^0\}$ with $s(e_x)=r(e_x)=x$ for all $x\in\dgraph^0$, $\alf=\{a\}$ and $\lbf:\dgraph^1\to\alf$ the only possible labeling function. For $\acf$ we take all finite subsets of $\nn$ and sets of the form $A\cup\{\infty\}$ where $A$ is a cofinite subset of $\nn$, that is, $\acf$ is the set of compact-open subsets of the one-point compactification of $\nn$. We have that $\awstar=\{a^n\mid n\in\nn\}$ and $\awinf=\{a^{\infty}\}$. Because $\dgraph$ has no sinks and $\alf$ is finite, $\ftight$ consists only of filters of infinite type, by Theorem~\ref{thm:TightFiltersType}. Moreover $\acfrg{a^n}=\acf$ for all $n\in\nn$. By Stone's duality the ultrafilters of $\acf$ are of the form $\ft_x=\{A\in\acf\mid x\in A\}$ for some $x\in\dgraph^0$. Also notice that for all $n\in\nn$ and all $A\in\acf$, we have that $(a^n,A)$ is a cycle without exits. This implies that if $\xi\in\ftight$ then its associated words is $a^\infty$ and there exists $x\in\dgraph^0$ such that $\xi_n=\ft_x$ for all $n\in\nn$. Also, the map that sends $\xi$ to such $x$ is a homeomorphism between $\ftight$ and $\nn\cup\{\infty\}$. As for the partial action, since there is a single letter, $\F=\mathbb{Z}$ and by the description of $\ftight$, we see that the partial action is actually the trivial action. This implies that $\F\ltimes_\varphi\ftight=\text{Iso}(\F\ltimes_\varphi\ftight)=\text{Iso}(\F\ltimes_\varphi\ftight)^{0}\cong \mathbb{Z}\times(\nn\cup\{\infty\})$ with the product topology. In this case, for every $t\in\mathbb{Z}$, the point $(t,\infty)$ is not isolated.
\end{example}

\begin{definition}(cf. \cite[Proposition-Definition~3.1]{saragab})
    We define the \emph{abelian core}, denoted by $M(L_R\lspace)$, as the $R$-subalgebra of $L_R\lspace$ spanned by elements of the form
    \[ s_\alpha p_A s_{\alpha}^*,\ s_\alpha p_C s_\gamma^* s_{\alpha}^*\text{ and }s_\alpha s_\gamma p_C s_{\alpha}^*,
    \]
    where $\alpha\in \awstar$, $A\in\acfra$ and $(\gamma,C)$ is a cycle without exits.
\end{definition}

Notice that $D(L_R\lspace)\scj M(L_R\lspace)$, because $M(L_R\lspace)$ contains the generators of $D(L_R\lspace)$.

\begin{proposition}\label{p:abcore}
Let $\Psi: L_R\lspace \to  A_R(\mathbb{F} \ltimes_{\varphi} \ftight)$ denote the isomorphism as in Theorem~\ref{thm:LeavittSteinbergIsom}. Then $\Psi(M(L_R\lspace)) =A_R(\text{Iso}(\F\ltimes_\varphi\ftight)^0)$.
\end{proposition}
    \begin{proof}
      Since $\F\ltimes_\varphi\ftight$ is an \'etale groupoid, we have that $A_R((\F\ltimes_\varphi\ftight)^{(0)})\scj A_R(\text{Iso}(\F\ltimes_\varphi\ftight)^0)$. Therefore, it follows from Proposition~\ref{prop:diag} that $\Psi(D(L_R\lspace))\scj A_R(\text{Iso}(\F\ltimes_\varphi\ftight)^0)$. Let $\alpha\in\awstar$ and suppose $(\gamma, C)$ is a cycle without exits. Using Equation~(\ref{eq:IsomOnTriples}), we have 
     
     \begin{equation} \label{eq:CoreIsom}
     \begin{aligned} 
         \Psi(s_\alpha p_C s_\gamma^*  s_{\alpha}^*) & = 1_{\{\alpha\gamma^{-1}\alpha^{-1}\}\times V_{(\alpha, C\cap r(\alpha)\cap r(\alpha\gamma),\alpha)}} \text{ and} \\
         \Psi(s_\alpha s_\gamma p_C s_{\alpha}^*) & = 1_{\{\alpha\gamma\alpha^{-1}\}\times V_{(\alpha\gamma, C\cap r(\alpha\gamma)\cap r(\alpha),\alpha\gamma)}}.
     \end{aligned}
     \end{equation}
     Since $(\gamma,C)$ is a cycle without exits, it follows from Lemma~\ref{l:intiso} that 
     \begin{align*}
        & \{\alpha\gamma^{-1}\alpha^{-1}\}\times V_{(\alpha, C\cap r(\alpha) \cap r(\alpha\gamma),\alpha)}\scj \text{Iso}(\F\ltimes_\varphi\ftight)^0 \text{ and}  \\
        &  \{\alpha\gamma\alpha^{-1}\}\times V_{(\alpha\gamma, C\cap r(\alpha\gamma)\cap r(\alpha),\alpha\gamma)}\scj \text{Iso}(\F\ltimes_\varphi\ftight)^0.
     \end{align*}
     Hence, $\Psi(M(L_R\lspace))\scj A_R(\text{Iso}(\F\ltimes_\varphi\ftight)^0)$.
     
      We show that $A_R(\text{Iso}(\F\ltimes_\varphi\ftight)^0)\scj \Psi(M(L_R\lspace))$. Let $U\scj \text{Iso}(\F\ltimes_\varphi\ftight)^0$ be a compact-open bisection and let us prove that $1_U\in \Psi(M(L_R\lspace))$. 
      Because $U$ is compact, there exists a finite set of elements $t\in\F$ such that $(\{t\}\times V_{t})\cap U\neq \emptyset$. Let $\{t_1,\ldots,t_n\}$ be this set. Then $1_U=\sum_{i=1}^n 1_{(\{t_i\}\times V_{t_i})\cap U}$, so we may assume without loss of generality that $U\scj \{t\}\times V_{t}$ for some $t\in\F$. If $t=\eword$, then $1_U\in A_R((\F\ltimes_\varphi\ftight)^{(0)})=\Psi(D(L_R\lspace))$. Suppose now that $t\neq\eword$. In order to show that $1_U\in \Psi(M(L_R\lspace))$, by Proposition~\ref{p:isobasis} and the inclusion-exclusion principle, it is sufficient to prove this for $U=\{t\}\times V_{(\beta,C,\beta)}$, where $t=\beta\gamma\beta^{-1}$ or $t=\beta\gamma^{-1}\beta^{-1}$, for some cycle without exists $(\gamma,C)$. Note that $C\scj r(\beta) \cap r(\beta\gamma)$. Thus, by Equations~(\ref{eq:CoreIsom}) and Lemma~\ref{lem:NeighbEqualAllN} we get $1_U=\Psi(s_\beta s_\gamma p_C s^*_{\beta})$ if $t=\beta\gamma\beta^{-1}$, or $1_U=\Psi(s_\beta p_C s^*_\gamma s^*_{\beta})$ if $t=\beta\gamma^{-1}\beta^{-1}$. In both cases, $1_U\in\Psi(M(L_R\lspace))$. Thus, $\Psi$ is onto and the proof is complete.
    \end{proof}
    
\begin{corollary}
    The core subalgebra $M_R\lspace$ of $L_R\lspace$ is maximal abelian.
\end{corollary}

\begin{proof}
Let $\xi\in\ftight$. By Remark~\ref{rmk:iso}, either the isotropy group associated with $\xi$ is trivial or there exists $\beta\in\awstar$ and $\gamma\in\awplus$ such that the isotropy group is $\{\beta\gamma^n\beta^{-1}\mid n\in\mathbb{Z}\}$. In both cases, we have an abelian group. The result then follows from Proposition~\ref{p:abcore} and \cite[Corollary~2.3]{HazratLi}.
\end{proof}

\begin{theorem}\label{GUTA}[Generalized uniqueness theorem]
    Let $\lspace$ be a normal labelled space and $R$ and commutative unital ring. Also, let  $\phi:L_R\lspace\to B$ be an $R$-algebra homomorphism. Then, $\phi$ is injective if and only if $\phi|_{M(L_R\lspace)}$ is injective.
\end{theorem}

\begin{proof}
This follows from \cite[Theorem~3.1]{CEP} and Proposition~\ref{p:abcore}.
\end{proof}

\section{Simplicity}\label{section:simplicity}

In this section, we use the groupoid realization of Leavitt labelled path algebras presented in Section~6 to describe simplicity criteria that depend on the base field of the algebra and on combinatorial properties of the associated labelled space (Theorem~\ref{theorem:simple.labelled.space.algebra}). Along the way, we characterize effectiveness of the associated groupoid in terms of Condition~($L_\acf$) and topological freeness of the associated partial action (Propositon~\ref{gpdeffect}). We also take the opportunity to characterize continuous orbit equivalence of topologically free partial actions associated with labelled spaces (Corollary~\ref{orb1qnob}). We finish the section with an application of the simplicity criteria, presenting two labelled graphs that represent the same edge shift space, but which possess non isomorphic associated Leavitt algebras (Example~\ref{tubarao}).

\begin{proposition}\label{gpdeffect}
Let $\lspace$ be a normal labelled space. The following are equivalent:
\begin{enumerate}[(i)]
    \item\label{i:ef1} The labelled space $\lspace$ satisfies Condition~($L_\acf$).
    \item\label{i:ef2} $\F\ltimes_\varphi\ftight$ is topologically principal.
    \item\label{i:ef3} $\F\ltimes_\varphi\ftight$ is effective.
    \item\label{i:eg} The partial action $\varphi$ of $\F$ on $\ftight$ is topologically free.
\end{enumerate}
\end{proposition}

\begin{proof}
The equivalence between \eqref{i:ef1} and \eqref{i:ef2} is shown in \cite[Theorem~6.13]{GillesDanie} and the equivalence between \eqref{i:ef2} and \eqref{i:eg} is shown in \cite[Proposition~6.3]{GillesDanie}. The implication \eqref{i:ef2}$\Rightarrow$\eqref{i:ef3} is in \cite[Proposition~3.6]{JeanCartan}. 

We prove \eqref{i:ef3}$\Rightarrow$\eqref{i:ef1}. Suppose that $\lspace$ does not satisfy condition ($L_\acf$), that is, there exists a cycle without exits $(\gamma,C)$. By Lemma~\ref{l:intiso}, $\text{Iso}(\F\ltimes_\varphi\ftight)^0$ contains elements other than units of $\F\ltimes_\varphi\ftight$, so that $\F\ltimes_\varphi\ftight$ is not effective.
\end{proof}

 Now that we have characterized topological freeness of the partial action associated with a labelled space, we describe continuous orbit equivalence of partial actions associated with labelled spaces in terms of diagonal preserving isomorphisms (we refer the reader to \cite{MR3743184} for the definition of continuous orbit equivalence between partial actions).

\begin{corollary}
    \label{orb1qnob}
    Let $\lspacei{i}$, i= 1,2, be two normal labelled spaces such that the associated partial actions $\varphi_i$ of the free group $\mathbb{F}_i$ on $\ftight_i$, $i=1,2$ are topologically free, and let $R$ be an integral domain. Then $\varphi_1$ is continuous orbit equivalent to $\varphi_2$ if, and only if, there is an isomorphism $\Phi:\Lc(\ftight_1,R)\rtimes_{\hat{\varphi_1}} \mathbb{F}_1 \rightarrow \Lc(\ftight_2,R)\rtimes_{\hat{\varphi_2}} \mathbb{F}_2 $ such that $\Phi (\Lc(\ftight_1, R)) = \Lc(\ftight_2, R) $.
\end{corollary}
\begin{proof}
This follows from the realization of a Leavitt labelled path algebra as a partial skew ring (Theorem~\ref{thm:SkewRingLeavittIsom}) and from \cite[Theorem~4.9]{MR3743184}.
\end{proof}

Before we present the simplicity criteria for Leavitt labelled path algebras, we recall the definition of hereditary and saturated subsets.

\begin{definition}
	Let $\lspace$ be a normal labelled space. A subset $H$ of $\acf$ is  \emph{hereditary} if the following conditions hold:
	\begin{enumerate}[(i)]
		\item $r(A,\alpha)\in H$ for all $A\in H$ and all $\alpha\in\awstar$,
		\item $A\cup B\in H$ for all $A,B\in H$,
		\item if $B\in\acf$ is such that $B\scj A$ for some $A\in H$, then $B\in H$.
	\end{enumerate}
	A hereditary set $H$ is \emph{saturated} if given $A\in\acf_{reg}$ such that $r(A,a)\in H$ for all $a\in\alf$, then $A\in H$.
\end{definition}

\begin{theorem}\label{theorem:simple.labelled.space.algebra}
	Let $\lspace$ be a normal labelled space and $R$ a commutative unital ring. Then $L_R\lspace$ is simple if and only if $R$ is a field, the labelled space satisfies condition $(L_\acf)$, and $\{\emptyset\}$ and $\acf$ are the only hereditary saturated subsets of $\acf$.
\end{theorem}

\begin{proof}

By \cite[Corollary~4.6]{CEM} (or \cite[Theorem~3.5]{Bensimple}) the Steinberg algebra (over a field) of an ample, Hausdorff groupoid is simple if, and only if, the groupoid is effective and minimal. 

The result now follows from the realization of $L_R\lspace$ as a Steinberg algebra (Theorem~\ref{thm:LeavittSteinbergIsom}),  Proposition~\ref{gpdeffect}, which characterizes effectiveness of  $\F\ltimes_\varphi\ftight$, and \cite[Theorem~6.15]{GillesDanie}, which characterizes minimality of $\F\ltimes_\varphi\ftight$.
\end{proof}

\begin{remark}
In \cite[Theorem~6.16]{GillesDanie}, a similar result characterizing simplicity of the C*-algebra of a labelled space is proved. There, $\F\ltimes_\varphi\ftight$ is also assumed second-countable in order to prove the implication that if the algebra is simple, then the labelled space satisfies condition $(L_\acf)$, and $\{\emptyset\}$ and $\acf$ are the only hereditary saturated subsets of $\acf$. Using Proposition~\ref{gpdeffect}, we can remove the second-countability assumption and obtain a result analogous to Theorem~\ref{theorem:simple.labelled.space.algebra} in the C*-algebraic case.
\end{remark}

\begin{example}\label{tubarao}
The two labelled graphs $\lgraphgi{E}{L}{1}$ and $\lgraphgi{E}{L}{2}$ of \cite[Example~3.3(iii)]{BP1} shown below represent the usual even shift of symbolic dynamics.

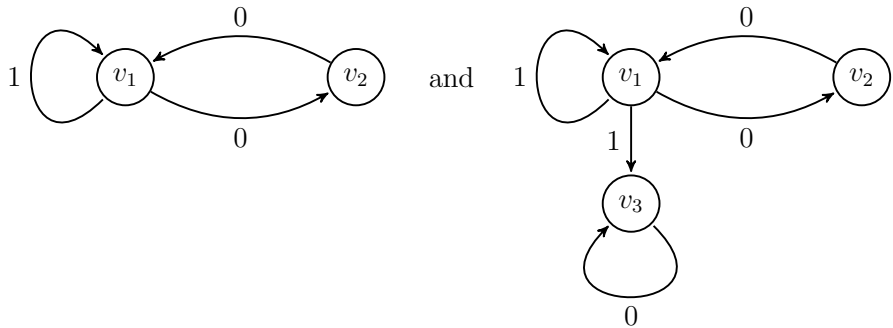
\resizebox{12cm}{!}{
		\begin{tikzpicture}[->,>=stealth',shorten >=1pt,auto,thick,main node/.style={circle,draw,thick,font=\bfseries}]
		\node[main node] (1) {$v_1$};
		\node[main node] (2) [right=2.5cm of 1] {$v_2$};
		\node[main node] (3) [below=1cm of 1] {$v_3$};
		
		\node (6) [left = 1.7cm of 1] {\text{and}};
		\node[main node] (4) [left=6.5cm of 1] {$v_1$};
		\node[main node] (5) [right=2.5cm of 4] {$v_2$}; 
		
		\Loop[dist=2cm,dir=WE,label=$1$,labelstyle=left](1)  
		\Loop[dist=2cm,dir=SO,label=$0$,labelstyle=below](3)
		
		\draw[->] (1) to node [left] {1} (3);
		\draw[->] (2) to [bend right] node [above] {0} (1);
		\draw[->] (1) to [bend right] node [below] {0} (2);
		
		\Loop[dist=2cm,dir=WE,label=$1$,labelstyle=left](4) 
		
		\draw[->] (5) to [bend right] node [above] {0} (4);
		\draw[->] (4) to [bend right] node [below] {0} (5);		
		\end{tikzpicture}
	}
	
In each case, consider the accommodating family of all sets of vertices. Even if the labelled paths are the same, the associated algebras are not isomorphic. Indeed, by Proposition~\ref{LLPALPA}, the Leavitt labelled path algebras are isomorphic to the Leavitt path algebras of the underlying graphs. Now, the first Leavitt path algebra is simple, but the second one has an ideal (for example, the ideal generated by $v_3$), and hence they cannot be isomorphic.
\end{example}

\bibliographystyle{abbrv}
\bibliography{ref}

\begin{thebibliography}{10}

\bibitem{MR2045419}
F.~Abadie.
\newblock On partial actions and groupoids.
\newblock {\em Proc. Amer. Math. Soc.}, 132(4):1037--1047, 2004.

\bibitem{AbrAraMol}
G.~Abrams, P.~Ara, and M.~Siles~Molina.
\newblock {\em Leavitt path algebras}, volume 2191 of {\em Lecture Notes in
  Mathematics}.
\newblock Springer, London, 2017.

\bibitem{ALP}
G.~Abrams, A.~Louly, E.~Pardo, and C.~Smith.
\newblock Flow invariants in the classification of leavitt path algebras.
\newblock {\em J. Algebra}, 333:202--231, 2011.

\bibitem{AAP}
G.~Abrams and G.~A. Pino.
\newblock The leavitt path algebra of a graph.
\newblock {\em J. Algebra}, 293:319--334, 2005.

\bibitem{AT11}
G.~Abrams and M.~Tomdorde.
\newblock Isomorphism and morita equivalence of graph algebras.
\newblock {\em Trans. Amer. Math. Soc.}, 363:3733--3767, 2011.

\bibitem{AGG}
P.~Ara, M.~A. Gonzalez-Barroso, and E.~Pardo.
\newblock Fractional skew monoid rings.
\newblock {\em J. Algebra}, 278:104--126, 2003.

\bibitem{ACHR}
G.~Aranda~Pino, J.~Clark, A.~an~Huef, and I.~Raeburn.
\newblock Kumjian-pask algebras of higher-rank graphs.
\newblock {\em Trans. Amer. Math. Soc.}, 365:3613--3641, 2013.

\bibitem{MR3614028}
T.~Bates, T.~M. Carlsen, and D.~Pask.
\newblock {$C^*$}-algebras of labelled graphs {III}---{$K$}-theory
  computations.
\newblock {\em Ergodic Theory Dynam. Systems}, 37(2):337--368, 2017.

\bibitem{BP1}
T.~Bates and D.~Pask.
\newblock {$C\sp *$}-algebras of labelled graphs.
\newblock {\em J. Operator Theory}, 57(1):207--226, 2007.

\bibitem{MR3506079}
V.~M. Beuter and D.~Gon\c{c}alves.
\newblock Partial crossed products as equivalence relation algebras.
\newblock {\em Rocky Mountain J. Math.}, 46(1):85--104, 2016.

\bibitem{MR3743184}
V.~M. Beuter and D.~Gon\c{c}alves.
\newblock The interplay between {S}teinberg algebras and skew rings.
\newblock {\em J. Algebra}, 497:337--362, 2018.

\bibitem{BoavaDeCastroMortari1}
G.~Boava, G.~G. de~Castro, and F.~de~L.~Mortari.
\newblock Inverse semigroups associated with labelled spaces and their tight
  spectra.
\newblock {\em Semigroup Forum}, 94(3):582--609, 2017.

\bibitem{BoavaDeCastroMortari2}
G.~Boava, G.~G. de~Castro, and F.~de~L.~Mortari.
\newblock {${\rm C}^*$}-algebras of labelled spaces and their diagonal {${\rm
  C}^*$}-subalgebras.
\newblock {\em J. Math. Anal. Appl.}, 456(1):69--98, 2017.

\bibitem{Gil3}
G.~Boava, G.~G. de~Castro, and F.~de~L.~Mortari.
\newblock Groupoid {M}odels for the {C}*-{A}lgebra of {L}abelled {S}paces.
\newblock {\em Bull. Braz. Math. Soc. (N.S.)}, 51(3):835--861, 2020.

\bibitem{CO55}
T.~M. Carlsen and E.~Ortega.
\newblock Algebraic {C}untz--{P}imsner rings.
\newblock {\em Proc. London. Math. Soc.}, 103(4):601--653, 2011.

\bibitem{COP}
T.~M. Carlsen, E.~Ortega, and E.~Pardo.
\newblock {$C^*$}-algebras associated to {B}oolean dynamical systems.
\newblock {\em J. Math. Anal. Appl.}, 450(1):727--768, 2017.

\bibitem{TokeRout}
T.~M. Carlsen and J.~Rout.
\newblock Orbit equivalence of higher-rank graphs.
\newblock {\em Arxiv}, arXiv:1903.07179 [math.OA], 2019.

\bibitem{CEM}
L.~O. Clark and C.~Edie-Michell.
\newblock Uniqueness theorems for {S}teinberg algebras.
\newblock {\em Algebr. Represent. Theory}, 18:907--916, 2015.

\bibitem{CEP}
L.~O. Clark, R.~Exel, and E.~Pardo.
\newblock A generalized uniqueness theorem and the graded ideal structure of
  {S}teinberg algebras.
\newblock {\em Forum Math.}, 30(3):533--552, 2018.

\bibitem{CFHL}
L.~O. Clark, J.~Fletcher, R.~Hazrat, and H.~Li.
\newblock {$\mathbb{Z}$}-graded rings as {C}untz-{P}imsner rings.
\newblock {\em J. Algebra}, 536:82--101, 2019.

\bibitem{GDD2}
G.~G. de~Castro, D.~Gon\c{c}alves, and D.~W. van Wyk.
\newblock Ultragraph algebras via labelled graph groupoids, with applications
  to generalized uniqueness theorems.
\newblock {\em J. Algebra}, 579:456--495, 2021.

\bibitem{GillesDanie}
G.~G. de~Castro and D.~W. van Wyk.
\newblock Labelled space {$C^*$}-algebras as partial crossed products and a
  simplicity characterization.
\newblock {\em J. Math. Anal. Appl.}, 491(1):124290, 35, 2020.

\bibitem{DE1}
M.~Dokuchaev and R.~Exel.
\newblock The ideal structure of algebraic partial crossed products.
\newblock {\em Proc. Lond. Math. Soc. (3)}, 115(1):91--134, 2017.

\bibitem{MR2419901}
R.~Exel.
\newblock Inverse semigroups and combinatorial {$C\sp \ast$}-algebras.
\newblock {\em Bull. Braz. Math. Soc. (N.S.)}, 39(2):191--313, 2008.

\bibitem{MR3699795}
R.~Exel.
\newblock {\em Partial dynamical systems, {F}ell bundles and applications},
  volume 224 of {\em Mathematical Surveys and Monographs}.
\newblock American Mathematical Society, Providence, RI, 2017.

\bibitem{goncalves_royer_2019}
D.~Gonçalves and D.~Royer.
\newblock Simplicity and chain conditions for ultragraph {L}eavitt path
  algebras via partial skew group ring theory.
\newblock {\em J. Aust. Math. Soc.}, 109(3):299--319, 2020.

\bibitem{reduction}
D.~Gonçalves and D.~Royer.
\newblock Representations and the reduction theorem for ultragraph {L}eavitt
  path algebras.
\newblock {\em J. Algebraic Combin.}, 53:505--526, 2021.

\bibitem{MR3938320}
D.~Gon\c{c}alves and D.~Royer.
\newblock Infinite alphabet edge shift spaces via ultragraphs and their {$\rm
  C^*$}-algebras.
\newblock {\em Int. Math. Res. Not. IMRN}, 2019(7):2177--2203, 2019.

\bibitem{Hazrat13}
R.~Hazrat.
\newblock The dynamics of leavitt path algebras.
\newblock {\em J. Algebra}, 384:242--266, 2013.

\bibitem{HazratLi}
R.~Hazrat and H.~Li.
\newblock A note on the centralizer of a subalgebra of {S}teinberg algebra.
\newblock {\em arXiv preprint arXiv:1912.01932}, 2020.

\bibitem{imanfar2017leavitt}
M.~Imanfar, A.~Pourabbas, and H.~Larki.
\newblock The {L}eavitt path algebras of ultragraphs.
\newblock {\em Kyungpook Math. J.}, 60(1):21--43, 2020.

\bibitem{Keimel}
K.~Keimel.
\newblock Alg\`ebres commutatives engendr\'{e}es par leurs \'{e}l\'{e}ments
  idempotents.
\newblock {\em Canadian J. Math.}, 22:1071--1078, 1970.

\bibitem{MR2974110}
M.~V. Lawson.
\newblock Non-commutative {S}tone duality: inverse semigroups, topological
  groupoids and {$C^\ast$}-algebras.
\newblock {\em Internat. J. Algebra Comput.}, 22(6):1250058, 47, 2012.

\bibitem{LindMarcus}
D.~Lind and B.~Marcus.
\newblock {\em n Introduction to Symbolic Dynamics and Coding}.
\newblock Cambridge University Press, 1995.

\bibitem{saragab}
G.~Nagy and S.~Reznikoff.
\newblock {Abelian core of graph algebras}.
\newblock {\em J. Lond. Math. Soc}, 85(3):889--908, 03 2012.

\bibitem{Nam}
T.~G. Nam and N.~D. Nam.
\newblock Purely infinite simple ultragraph {L}eavitt path algebras.
\newblock {\em arXiv:2007.08144 [math.RA]}, 2020.

\bibitem{JeanCartan}
J.~Renault.
\newblock Cartan subalgebras in {$C^*$}-algebras.
\newblock {\em Irish Math. Soc. Bull.}, 61:29--63, 2008.

\bibitem{RTh}
E.~Ruiz and M.~Tomdorde.
\newblock Ideal-related {K}-theory for {L}eavitt path algebras and graph
  {C}*-algebras.
\newblock {\em Indiana Univ. Math J.}, 62:1587--1620, 2013.

\bibitem{BenGroupoid}
B.~Steinberg.
\newblock A groupoid approach to discrete inverse semigroup algebras.
\newblock {\em Adv. Math.}, 223(2):689--727, 2010.

\bibitem{BenModule}
B.~{Steinberg}.
\newblock {Modules over \'etale groupoid algebras as sheaves}.
\newblock {\em {J. Aust. Math. Soc.}}, 97(3):418--429, 2014.

\bibitem{Bensimple}
B.~Steinberg.
\newblock Simplicity, primitivity and semiprimitivity of étale groupoid
  algebras with applications to inverse semigroup algebras.
\newblock {\em J. Pure Appl. Algebra}, 220:1035--1054, 2016.

\end{thebibliography}

\end{document}